\documentclass[a4paper]{article}
\usepackage[all]{xy}
\usepackage[dvips]{graphics,graphicx}
\usepackage{amsfonts,amssymb,amsmath,mathrsfs, amstext}
\usepackage{amsbsy, amsopn, amscd, amsxtra, amsthm,authblk}
\usepackage{enumerate,algorithmic,algorithm}
\usepackage{upref}
\usepackage{geometry}
\usepackage{lineno}
\usepackage{float,subfig}
\usepackage{xcolor,pgf,tikz,pgfplots}
\usetikzlibrary{positioning,arrows.meta}
\usepackage{yhmath}
\usepackage{booktabs}
\usepackage{multirow}
\usepackage{makecell}
\usepackage{ulem}
\usepackage{enumitem}
\usepackage{bbm,bm,cases,color}
\usepackage[colorlinks,
            linkcolor=red,
            anchorcolor=red,
            citecolor=red
            ]{hyperref}

\numberwithin{equation}{section}  

\def \pt {\partial}
\def \eq$#1${\begin{equation}#1\end{equation}}
\def \F {\mathcal{F}}
\def \E {\mathbb{E}}

\def \I {\mathbbm{1}}
\newcommand{\dd}{\textup{d}}
\newcommand{\blue}[1]{{\color{black}#1}}

\newcommand{\norm}[1]{\left\lVert#1\right\rVert}

\makeatletter
\newenvironment{breakablealgorithm}
{
	\begin{center}
		\refstepcounter{algorithm}
		\hrule height.8pt depth0pt \kern2pt
		\renewcommand{\caption}[2][\relax]{
			{\raggedright\textbf{\ALG@name~\thealgorithm} ##2\par}%
			\ifx\relax##1\relax 
			\addcontentsline{loa}{algorithm}{\protect\numberline{\thealgorithm}##2}%
			\else 
			\addcontentsline{loa}{algorithm}{\protect\numberline{\thealgorithm}##1}%
			\fi
			\kern2pt\hrule\kern2pt
		}
	}{
		\kern2pt\hrule\relax
	\end{center}
}
\makeatother

\def \tpsi {\tilde{\psi}}

\def \tq {\tilde{q}}

\def \br {\bm{r}}
\def \bn {\bm{n}}

\def \tpsi {\tilde{\psi}}
\def \tphi {\tilde{\phi}}
\def \bphi {\bar{\phi}}

\def \K {\mathcal{K}}
\def \D {\mathcal{D}}

\def \R {\mathbb{R}}
\def \P {\mathbb{P}}

\def \S {\mathbb{S}}

\newtheorem{theorem}{Theorem}[section]

\newtheorem{lemma}[theorem]{Lemma}
\newtheorem{remark}[theorem]{Remark}

\title{Random Source Iteration Method: Mitigating the Ray Effect in the Discrete Ordinates Method}

 \author[1]{Jingyi Fu \thanks{nbfufu@sjtu.edu.cn} }
\author[1,2,3]{Lei Li \thanks{leili2010@sjtu.edu.cn} }
\author[1,2]{Min Tang \thanks{tangmin@sjtu.edu.cn} }

\affil[1]{School of Mathematical Sciences, Shanghai Jiao Tong University, Shanghai, 200240, P.R.China}
\affil[2]{Institute of Natural Sciences, MOE-LSC, Shanghai Jiao Tong University, Shanghai, 200240, P.R.China}
\affil[3]{Shanghai Artificial Intelligence Laboratory}

\date{}

\begin{document}
	
	\maketitle

\begin{abstract}
The commonly used velocity discretization for simulating the radiative transport equation (RTE) is the discrete ordinates method (DOM). One of the long-standing drawbacks of DOM is the phenomenon known as the ray effect. Due to the high dimensionality of the RTE, DOM results in a large algebraic system to solve. The Source Iteration (SI) method is the most standard iterative method for solving this system. In this paper, by introducing randomness into the SI method, we propose a novel random source iteration (RSI) method that offers a new way to mitigate the ray effect without increasing the computational cost. We have rigorously proved that RSI is unbiased with respect to the SI method and that its variance is uniformly bounded across iteration steps; thus, the convergence order with respect to the number of samples is $1/2$. Furthermore, we prove that the RSI iteration process, as a Markov chain, is ergodic under mild assumptions. Numerical examples are presented to demonstrate the convergence of RSI and its effectiveness in mitigating the ray effect.
\end{abstract}

	\section{Introduction}
The linear radiative transport equation (RTE) is widely used to model the propagation of particles or radiation through various media. It finds significant applications across a range of disciplines, including astrophysics, nuclear engineering, and thermodynamics. The steady-state anisotropic RTE takes the following form:
    \begin{subequations}\label{3dori}
    \begin{numcases}{}
		& $\Omega \cdot \nabla \psi(\br,\Omega)+\Sigma_T(\br)\psi(\br,\Omega)=\Sigma_S(\br)\phi(\br,\Omega)+Q(\br),\quad \br \in \D , \quad \Omega\in \S^2,$\label{3deqpsiori}\\
		&$\psi(\br,\Omega)=\psi_{\pt \D}^-(\br,\Omega), \quad\br\in\pt \D,\quad \Omega\cdot \bn_{\br}<0,$\label{3dbcori}\\
&$\displaystyle\phi(\br,\Omega)=\int_{\S^2}\K(\Omega,\Omega')\psi(\br,\Omega')\dd \Omega'. $\label{3deqphiori}
	\end{numcases}
    \end{subequations}
	Here $\br$ denotes the spatial variable within a bounded convex domain $\D \subset \R^3$, and $\Omega$ represents the angular variable on the unit sphere $\S^2$. The function $\psi(\br,\Omega)$ is the probability density function of particles moving in direction $\Omega$ at position $\br$. The parameters $\Sigma_T$, $\Sigma_S$ and $Q$ correspond to the total cross-section, scattering cross-section, and the source term, respectively. The scattering kernel $\K(\Omega,\Omega')$ characterizes the probability that particles traveling in direction $\Omega'$ are scattered into direction $\Omega$; $\bn_{\br}$ is the outward-pointing normal vector at the boundary $\br\in\pt \D$.

    There are two main categories of numerical methods for solving the RTE: particle methods and PDE-based methods. The particle methods (or Monte Carlo (MC) methods) replicate the behavior of individual particles as they interact with the medium, undergoing scattering or absorption events \cite{spanier1959monte,LewisMiller}. These methods are typically more straightforward to implement and can be parallelized for high-dimensional problems involving irregular and complex geometries. However, they come with a high computational cost, particularly for steady-state problems, and the statistical noise is unavoidable. On the other hand, deterministic methods are based on PDE solvers \cite{colomer2013parallel, dedner2002adaptive, giani2016hp, centuryreview}. They can outperform MC methods in terms of speed when addressing problems with simpler geometries. Nevertheless, they require mesh dissection or other preprocessing steps, and parallelization becomes more challenging for irregular or complex geometries.

   Note that \eqref{3dori} is a differential-integral equation, where particles moving in different directions are coupled together by the integral term in \eqref{3deqphiori}. The discrete ordinates method (DOM), commonly referred to as the $S_N$ method, is the most popular angular discretization technique. The DOM represents the solution through the values of $\psi(\br,\Omega_m)$ at discrete angular directions $\Omega_m$, with the integral term approximated by a weighted sum of these values. The DOM is favored due to its positive density fluxes and the ease of applying boundary and interface conditions \cite{LewisMiller, centuryreview}. Assuming there are $M$ discrete ordinates $\Omega_1$, $\Omega_2$, $\cdots$, $\Omega_M$, with respective weights $\omega_1$, $\omega_2$, $\cdots$, $\omega_M$, then the DOM system for $\psi_m(\br)\approx\psi(\br,\Omega_m)$ ($m\in V=\{1,2\cdots,M\}$) is written as
\begin{subequations}\label{3ddom}
	\begin{numcases}{}
		&$\Omega_m \cdot \nabla \psi_m(\br)+\Sigma_T(\br)\psi_m(\br)=\Sigma_S(\br)\phi_m(\br)+Q(\br),\quad \br \in \D  , \quad m\in V,$ \label{3deqpsi}\\
		&$\psi_m(\br)=\psi_{\pt \D}^-(\br,\Omega_m),\quad \br \in \pt\D, \quad \Omega_m\cdot \bn_{\br}<0,$ \label{3dbc}\\
		&$\displaystyle\phi_m(\br)=\sum_{k\in V} \omega_k K_{km}\psi_k(\br),  $\label{3deqphi}
	\end{numcases}
	\end{subequations}
where $\phi_m(\br)\approx \phi(\br,\Omega_m)$, $K_{km}\approx \K(\Omega_m,\Omega_k)$.
To ensure mass conservation at the discrete level, the discrete scattering kernel  $K_{km}$ and weights $\omega_m$ must satisfy the following properties:
    \eq$\label{assK} \sum_{k\in V}\omega_k K_{km}=1, \ \forall m \in V; \quad \sum_{m\in V}\omega_m K_{km}=1, \ \forall k \in V; \quad \sum_{m\in V}\omega_m=1.$

 The DOM has long been known to have a significant drawback: when the number of discrete ordinates is not sufficiently large, it can produce nonphysical spatial oscillations along the discrete angular paths in the macroscopic particle density, $\sum_{k\in V}\omega_k\psi_k(\br)$. These nonphysical oscillations are independent of the spatial discretization and cannot be mitigated by refining the spatial meshes. This phenomenon is referred to as the "ray effect" \cite{lathrop1968ray}.  This issue is particularly pronounced in cases of weak scattering, \blue{and when $Q(\br)$ exhibits discontinuities,} or inflow boundary conditions. 
    For high-dimensional problems, the convergence order of DOM in the angular variable is notably low, especially when the solutions exhibit the ray effect or for problems characterized by strongly anisotropic scattering \cite{ROM}. To capture an accurate solution, an extremely large number of discrete ordinates is required, which can be computationally prohibitive.

The classical spatial discretizations of the DOM system \eqref{3ddom} include diamond-difference methods, weighted-difference methods, nodal methods, discontinuous finite element methods \cite{morel1996linear, centuryreview}, and more recently, the discontinuous Galerkin method \cite{liu2010analysis, sheng2021uniform}. Regardless of the method chosen for spatial discretization, a large algebraic system must be inverted due to the high dimensionality of the equation.

Due to the coupling of different angular directions, the most widely used iterative approach for solving the resulting large system is the source iteration method (SI).  In the $n$th iteration step of SI, the density fluxes and scattering sources are denoted as $\psi_m^{(n)}(\br)$ and $\phi_m^{(n)}(\br)$, respectively. The SI method uses the scattering source $\phi_m^{(n-1)}(\br)$  obtained in the previous iteration to update the density flux $\psi_m^{(n)}(\br)$ using equations \eqref{3deqpsi} and \eqref{3dbc}, and subsequently derives a new scattering source $\phi_m^{(n)}(\br)$  via equation \eqref{3deqphi}.
    More precisely, beginning with an initial estimate $\phi_m^{(0)}(\br)$, the system is iteratively solved as follows:
    \begin{subequations}\label{3dsi}
       \begin{numcases}{}
		&$\Omega_m \cdot \nabla \psi_m^{(n)}(\br)+\Sigma_T(\br)\psi_m^{(n)}(\br)=\Sigma_S(\br)\phi_m^{(n-1)}(\br)+Q(\br),\quad \br \in \D  , \quad m\in V,$ \label{3deqpsisi}\\
		&$\psi_m^{(n)}(\br)=\psi_{\Gamma}^-(\br,\Omega_m),\quad \br \in \pt \D, \quad \Omega_m\cdot \bn_{\br}<0.$ \label{3dbcsi}\\
		&$\displaystyle\phi_m^{(n)}(\br)=\sum_{k\in V} \omega_k K_{km}\psi_k^{(n)}(\br).  $\label{3deqphisi}
	\end{numcases} 
    \end{subequations}

    In this paper, we propose a random source iteration (RSI) method that offers a new way to mitigate the ray effect without increasing the computational cost. The idea is to introduce randomness into the SI method. 
More precisely, let  $M=|V|$ represent the number of ordinates of a fixed quadrature set. From $V$,
    $G\ll M$ ordinates are randomly selected, with their indices forming a subset $V^{(n-1)}\subset V$. Then the scattering sources $\tilde\phi_{m}^{(n-1)}$ for a newly chosen subset of ordinates with $m\in V^{(n)}$ are approximated by a weighted function of $\tilde\psi^{(n-1)}_{k}$ ($k\in V^{(n-1)}$). The spatial transport equation with the scattering source $\tilde\phi_{m}^{(n-1)}$ ($m\in V^{(n)}$) is then solved for $m\in V^{(n)}$ to get $\tilde\psi_{m}^{(n)}$ ($m\in V^{(n)}$). 
    
    Note that only $G$ ordinates are selected in each iteration step, which significantly reduces the computational cost compared to the original SI.  When $G=1$,  the method selects just one ordinate per iteration step, eliminating the need for collecting information from other ordinates and enabling perfect parallelism in the angular variable. We refer to one execution of RSI as one sample, with all samples being independent of each other.  One might be concerned about using $G$ ordinates $\tilde\psi^{(n-1)}_{k}$ ($k\in V^{(n-1)}$) to approximate the summation $\displaystyle\phi_m^{(n-1)}(\br)=\sum_{k\in V} \omega_k K_{km}\psi_k^{(n-1)}(\br)$ ($m\in V^{(n)}$). However, thanks to the randomized process and the careful selection of $V^{(n)}$ and the weights, we can rigorously prove that the expectation of all samples converges to the solution of SI, and the variance is uniformly bounded across iteration steps. Furthermore, we prove that the RSI iteration process, as a Markov chain, is ergodic under mild assumptions, indicating the ergodicity of the process. Consequently, all tail data after SI converges can be used to compute the expectation.

The RSI method can be viewed as deterministic in the spatial variable while employing a Monte Carlo approach for the angular variable. On one hand, RSI inherits the Monte Carlo method's advantage of easy parallelization, allowing each sample of RSI to be run on different processors in parallel. The expectations of these samples can then be taken to obtain the SI solution. On the other hand, RSI reduces computational cost and statistical noise through the use of a spatial deterministic solver. Numerical results show that the convergence rate of RSI is $1/2$ with respect to the number of samples $S$, which is identical to the Monte Carlo method. \blue{Moreover, since SI only converges when $\Sigma_S<\Sigma_T$ uniformly in space, this requirement is also necessary for RSI.}

We can compare the computational complexities of SI and RSI. Suppose we use $I$ spatial grids and $M$ discrete ordinates in the discretization process. In each iteration, updating the $I\times M$ values of $\phi_m^{(n)}$ as in equation \eqref{3deqphisi} requires in $I\times M^2$ multiplications. Subsequently, the sweeping method is applied to solve the transport operator for each angular direction. The computational cost for performing this sweeping method is $O(I\times M)$ \cite{centuryreview, GaoZhao2009}.  Assuming the DOM has first-order convergence, achieving $O(\epsilon)$ accuracy in the velocity variable requires $M=O(\epsilon^{-1})$ ordinates. Equation \eqref{3dsi} indicates that the total computational cost for SI is $(I\times \epsilon^{-2}+CI\times \epsilon^{-1})\times N_{\text{iter}}$, where $N_{\text{iter}}$ is the number of iterations and $C$ is a constant.
Since the variance of RSI is bounded, to attain $O(\epsilon)$ accuracy, the necessary number of samples is $O(\epsilon^{-2})$. Each sample performs one sweep per iteration step, resulting in a complexity of $O(I\times \epsilon^{-2})$ for $O(\epsilon^{-2})$ samples. Given that  RSI is unbiased w.r.t. SI for each iteration step, the overall complexity is $O(I\times \epsilon^{-2}\times N_{\text{iter}})$. Therefore, when the DOM exhibits first-order convergence in the angular variable, the overall computational cost to achieve the same level of accuracy is identical for both RSI and SI.

 There are two benefits of RSI: 1) When the convergence order of DOM in the velocity variable is less than 1, RSI's overall complexity is lower than that of SI to achieve the same accuracy. This is particularly relevant for many benchmark tests that exhibit ray effects, as we will demonstrate with two test cases in the numerical section. 2) When the number of discrete ordinates $M$ is small, DOM can exhibit the ray effect. Conversely, assuming an error tolerance $\epsilon$, RSI can be applied to a DOM system with significantly more than $\epsilon^{-1}$ ordinates ($M\gg \epsilon^{-1}$), running $\epsilon^{-2}$ samples and taking the expectation. Since $M\gg \epsilon^{-1}$, the error of the mean is much smaller than $\epsilon$,  so the error primarily comes from the sample average, resulting in $O(\epsilon)$  accuracy. Consequently, RSI offers accuracy similar to SI but mitigates the ray effect since more ordinates are used. As a consequence of the ergodicity, an extra benefit is that all tail data in the RSI can be used as samples for the computation of the expectation. Hence, the number of chains might be even smaller, resulting in a lower cost to achieve the same level of accuracy.

In recent years, several attempts have been made to mitigate the ray effect. One popular approach is the first collision source method, \cite{Lathrop1971}, which decomposes the solution into collided and uncollided parts. The uncollided part is solved using ray tracing. However, the complexity of this method increases with the addition of more sources and the presence of reflective boundary conditions \cite{wareing1998first,Dai2023}. The region angular adaptive algorithm has also been developed, based on angular adaptive refinement techniques that require local error estimation \cite{zhang2018goal}. The reference frame rotation method \cite{Tencer2016} averages solutions for various reference frame orientations to mitigate the ray effect; however, it is limited to isotropic scattering. The rotational summation method \cite{camminady2019ray} solves the time-dependent RTE, rotating the set of ordinates around a random axis after each time step. Since additional technical steps are usually employed to mitigate the ray effect, rigorous error analysis is rare for the aforementioned methods. The concept of introducing randomness into velocity discretization has been explored in several previous works. The random ordinate method (ROM), discussed in \cite{ROM}, randomly selects one ordinate from each cell in the velocity space, solves the resulting DOM systems, and then averages the solutions. Our approach, which is based on SI, shares similarities with the \blue{quasi-random discrete ordinates method (QRDOM) proposed in \cite{de2019quasi, pasmann2023quasi,Konzen2023}, but also differs from it. Specifically, our approach, due to the specific method of selecting the new ordinate for the next iteration step, requires only one ordinate per iteration, whereas QRDOM requires more.} Moreover, we provide rigorous analytical results for the properties of RSI.

An outline of the remainder of this article is as follows: In Section \ref{sec:rsi}, we introduce the RSI method and present its main analytical results. In Section \ref{sec:proof}, we rigorously prove the stochastic properties of RSI, including the expectation that yields the solution of the SI, uniformly bounded variance, contractility, and ergodicity. In Section \ref{sec:numer}, we present numerical examples that test and confirm our theoretical results, demonstrating the performance of the RSI method. We conclude with a discussion in Section \ref{sec:discussion}.

	\section{RSI method and its properties}\label{sec:rsi}
	
\subsection{The RSI method}
We present the RSI for the steady state RTE based on SI as in \eqref{3deqphisi}.
We partition the set $V=\{1,2\cdots,M\}$ into $G$ groups, denoted as $V_1,V_2,\cdots,V_G$. The number of ordinates in the $g$-th group is represented by $|V_g|$ ($g=1,2,\cdots, G$). In the most extreme case where
$G=1$, the computational complexity of each sample is minimized. In the first step, we select one ordinate from each group with uniform probability and denote the index set of these chosen discrete ordinates by $V^{(0)}$.  We then initialize $\tpsi_m^{(0)}(\br)=\psi_m^{(0)}(\br)$ for $m\in V^{(0)}$, and compute $\tphi_m^{(0)}(\br)$ by the equation
\eq$\tphi_m^{(0)}(\br)= \sum_{k\in V^{(0)}}\omega_k K_{km}  |V(k)|\tpsi_{k}^{(0)}(\br),\label{tildephi0}$
where $V(k)$ is the group to which index 
$k$ belongs. 
For subsequent steps, we proceed by induction. Assuming the index set chosen in the $n-1$th iteration  is $V^{(n-1)}$, with $|V^{(n-1)}|=G$,
in the $n$th step, we define $c_m^{(n-1)}$ by
\eq$c_m^{(n-1)}=\displaystyle\sum_{k\in V^{(n-1)}}\omega_k K_{km},\quad \forall m\in V.\label{cmn}$
Then, one ordinate is randomly chosen from each  $V_g(g=1,2,\cdots,G)$ to form a subset $V^{(n)}$ of $V$, following the rule:
	\eq$\P\left(m\in V^{(n)}\bigg|V^{(n-1)}\right)=\frac{\omega_m c_m^{(n-1)}}{\sum_{m'\in V(m)}\omega_{m'} c_{m'}^{(n-1)}}:=p_{m}^{(n)}.\label{qkn}$
It is clear that for all $g=1,2,\cdots,G$, 
\begin{equation}
\sum_{m\in V_g}p_{m}^{(n)}=1,\label{qVg}
\end{equation}
 which is consistent with the fact that one index is chosen from each group. Moreover, the selection of the ordinate from each group is independent of the others.
 If $p_m^{(n)}=0$ or $c_m^{(n-1)}=0$, given the positivity of $\omega_k$ and the nonegativity of $K_{km}$ (for $\forall k,m\in V$),  it follows that  $K_{km}=0$ for all $k\in V^{(n-1)}$. This implies that particles moving in direction $\Omega_k$ (for all $ k\in V^{(n-1)}$) can not be scattered into direction $\Omega_m$, and thus $\Omega_m$ cannot be selected for the next iteration step.

One might be concerned about the computational cost of determining $p_m^{(n)}$.  In practice, to reduce computational cost at each iteration, since
  $$
  \sum_{m'\in V(m)}\omega_{m'} c_{m'}^{(n-1)}=\sum_{m'\in V(m)}\omega_{m'} \sum_{k\in V^{(n-1)}}\omega_k K_{km'}=\sum_{k\in V^{(n-1)}}\omega_k\sum_{m'\in V(m)}\omega_{m'}K_{km'},
  $$
One can pre-compute $\omega_k K_{km}$ for all $ k,m\in V$, $\omega_k\sum_{m\in V_g}\omega_m K_{km}$ for all $ g=1,\cdots,G$ and all $ k\in V$. Then both $c_m^{(n-1)}$ in \eqref{cmn} and the denominator in \eqref{qkn} can be obtained by $G-1$ summations. Then $O(GM)$ summations are needed to obtain $p_m^{(n)}$ which increases little computational cost.

    Then we update the density flux function $\tpsi_k^{(n)}$ according to equations \eqref{3deqpsisi}-\eqref{3dbcsi} with the scattering source $\tphi_k^{(n-1)}(\br)$ only for $k\in V^{(n)}$, and the new scattering sources for the next iteration step are needed only for $m\in V^{(n+1)}$. $\tphi_m^{(n)}(\br)$ ($m\in V^{(n+1)}$) are determined by a weighted sum of these density functions:
	\eq$\tphi_m^{(n)}(\br)=\sum_{k\in V^{(n)}}\omega_k K_{km} \frac{1}{p_{k}^{(n)}}\tpsi_{k}^{(n)}(\br),\qquad m\in V^{(n+1)}.\label{1drsi0}$ 
	Here, since the probability that index $k$ is chosen in $V^{(n)}$ is $p_{k}^{(n)}$, we introduce an extra weight $1/p_{k}^{(n)}$, ensuring that $\frac{\tpsi_{k}^{(n)}}{p_{k}^{(n)}}$ gives the normalized density flux of $V(m)$.
	Note that when $p_{k}^{(n)}=0$, we have $k\notin V^{(n)}$. For convenience, we define 
	\eq$\tq_{k}^{(n)}=\left\{
	\begin{aligned}
	&0, &&p_{k}^{(n)}=0, \\
	&\frac{1}{p_{k}^{(n)}}, &&p_{k}^{(n)}\neq 0, 
	\end{aligned}
	\right.\label{deftq}$
	then \eqref{1drsi0} can be rewritten as
	\eq$\tphi_m^{(n)}(\br)=\sum_{k\in V}\omega_k K_{km} \tq_{k}^{(n)}\tpsi_{k}^{(n)}(\br)\I_{k\in V^{(n)}},\qquad m \in V^{(n+1)}.\label{1drsi}$

	The RSI method reads:
	\begin{breakablealgorithm}
		\caption{Random Source Iteration}
		\label{SIR}
		\begin{algorithmic}[1]
			\STATE Choose one ordinate from each group with uniform probability to form $V^{(0)}$.
            \STATE Initialize $\tpsi_m^{(0)}(\br)=\psi_m^{(0)}(\br)$, for $m\in V^{(0)}$.
			\FOR{$n=1,2,\cdots$}
			\STATE Pick a subset $V^{(n)}$ from $V$ according to rule \eqref{qkn}.
			\STATE Compute $\tphi_{m}^{(n-1)}(\br)$ by \eqref{1drsi} with $\tpsi_{k}^{(n-1)}(\br) (k\in V^{(n-1)})$ for $m\in V^{(n)}$.
			\STATE Solve $\tpsi_{m}^{(n)}(\br)$ by \eqref{3deqpsisi}-\eqref{3dbcsi} with scattering source $\tphi_m^{(n-1)}(\br)$ for $m\in V^{(n)}$. 
			\ENDFOR
		\end{algorithmic}
	\end{breakablealgorithm}


\subsection{Main theoretical results}
In each iteration step, since $V^{(n)}$ is chosen randomly, $\tpsi_{m}^{(n)}(\br)$ is also random. We refer to one realization of RSI as one sample.
One might be concerned with the use of $k\in V^{(n)}$ to update the sources as in \eqref{1drsi}; however, we are able to prove the following properties of RSI:
\begin{enumerate}
    \item RSI is unbiased with respect to SI. By taking the expectation of all samples, the density flux $\tpsi_m^{(n)}(\br)$ and scattering source $\tphi_m^{(n)}(\br)$ in RSI yield the values of $\psi_m^{(n)}$ and $\phi_m^{(n)}$ obtained by SI.
    \item The variance of RSI is uniformly bounded with respect to the iteration steps $n$. Owing to this uniform boundedness of variance, one can estimate the required number of samples to achieve a given accuracy after computing the numerical expectation.
    \item The average of all iterations in RSI converges to the solution of SI. This suggests that one can run each RSI sample for more iteration steps, utilize data from the tail end, and then take the expectation to obtain the solution of SI. This offers the flexibility to terminate the iteration of RSI.
\end{enumerate}

Before stating the main theorems of this paper, we introduce some notations. Consider a state
 $\bm{\tpsi}$ in RSI, defined as
\begin{gather}
\bm{\tpsi}:=(\tpsi_m)_{m\in V}=(\tpsi_1,\cdots, \tpsi_M).
\end{gather}
Let  $\norm{\cdot}_2=\left(\int_{\D}(\cdot)^2\dd \br \right)^{\frac{1}{2}}$. Then, a state can be viewed as a point in $S_{\psi}:=(L^2(\D; \Sigma_T))^{\otimes M}$,  where $L^2(\D; \Sigma_T)$ denotes the space of square-integrable functions with weight $\Sigma_T$. 
We define the distance between two such states by
\begin{equation}
\mathrm{dist}(\bm{\tpsi}_1, \bm{\tpsi}_2)
:=\sum_{m\in V}\omega_m \|\tpsi_{m,1}-\tpsi_{m,2}\|_{L^2(\D; \Sigma_T)}
=\sum_{m\in V}\omega_m \|\sqrt{\Sigma_T}(\tpsi_{m,1}-\tpsi_{m,2})\|_{2}.
\end{equation}
In RSI, each value of $V^{(n)}$ can be represented as a vector $v\in \{0, 1\}^M=: \mathcal{S}_V$ with $\sum_{i\in V_g} v_i=1$ for each $g=1,\cdots, G$.
Clearly, the RSI forms a Markov chain on the state space $\mathcal{S}:=\mathcal{S}_V\times \mathcal{S}_{\psi}$. Moreover, $V^{(n)}$ itself is also a Markov chain on $\mathcal{S}_V$ since its evolution does not depend on $\bm{\psi}$.

Consider two probability measures $\mu_i\in \mathcal{P}(\mathcal{S})$ ($i=1,2$) where $\mathcal{P}(\mathcal{S})$ denotes the set of all probability measures on $S$. What we care about is the law of $\bm{\tilde{\psi}}$, which is the marginal distribution on $\mathcal{S}_{\psi}$ 
\begin{equation}
\mu_i^{\psi}:=\int_{\mathcal{S}_V}\mu_i(dv, \cdot)\in \mathcal{P}(\mathcal{S}_{\psi}).
\end{equation}
We use the Wasserstein-1 distance \cite{santambrogio2015optimal}, given by
\begin{equation}\label{eq:Wa1}
W_1(\mu_1^{\psi}, \mu_2^{\psi})=\inf_{\gamma\in \Pi(\mu_1^{\psi}, \mu_2^{\psi})} \int_{\mathcal{S}_{\psi} \times \mathcal{S}_{\psi}} \mathrm{dist}(\bm{\tpsi}_1, \bm{\tpsi}_2) d\gamma(\bm{\tpsi}_1, \bm{\tpsi}_2)
\end{equation}
to gauge the closeness of two measures in $\mathcal{P}(\mathcal{S}_{\psi})$.
Here, $\Pi(\mu_1, \mu_2)$ is the so-called set of transport plans, and each transport plan $\gamma \in \Pi(\mu_1, \mu_2)$ is a joint distribution on $\mathcal{S}{\psi} \times \mathcal{S}{\psi}$, with marginals that are $\mu_1^{\psi}$ and $\mu_2^{\psi}$, respectively.
With the above notations,  we present the main results of this paper.

\begin{theorem}(unbiased w.r.t SI)\label{expectation}
		Suppose $V^{(n)}$ is chosen according to \eqref{qkn}, then $\tpsi_m^{(n)}$ and $\tphi_m^{(n)}$ are unbiased estimates of $\psi_m^{(n)}$ and $\phi_m^{(n)}$, i.e. \eq$\E\left[\tpsi_m^{(n)}\right]\equiv\psi_m^{(n)}, \quad \E\left[\tphi_m^{(n)}\right]\equiv\phi_m^{(n)}.$
	\end{theorem}

\begin{theorem}(uniformly boundedness of the variance)\label{variance}
	Assume that $\omega_k>0$, $K_{km}>0$, then $p_m^{(n)}$ in \eqref{qkn} are strictly positive. Let 
	\eq$\mathcal{M}_{m}^{(n)}=\E\left[\frac{1}{2}\int_{\D}\Sigma_T \sum_{g=1}^G\sum_{k\in V_g}\sum_{k'\in V_g}\left(\sqrt{\frac{p_{k'}^{(n)}}{p_{k}^{(n)}}}\omega_k K_{km}\tpsi_k^{(n)}-\sqrt{\frac{p_{k}^{(n)}}{p_{k'}^{(n)}}}\omega_{k'} K_{k'm}\tpsi_{k'}^{(n)}\right)^2 \dd \br\right],\label{assM}$
  and $\mathcal{M}=\max_{m,n}\{M_m^{(n)}\}$.
Then for any $n$, the variances $$ \E\left[\sum_{m \in V}\omega_m\norm{\sqrt{\Sigma_T}\left(\tpsi_m^{(n)}-\psi_m^{(n)}\right)}_2^2\right], \quad \text{and}\quad \E\left[\sum_{m \in V}\omega_m\norm{\sqrt{\Sigma_T}\left(\tphi_m^{(n)}-\phi_m^{(n)}\right)}_2^2\right]$$
are bounded by  
  \eq$\norm{\frac{\Sigma_S}{\Sigma_T}}_{\infty}^{2n}V_0+\frac{1-\norm{\frac{\Sigma_S}{\Sigma_T}}_{\infty}^{2n}}{1-\norm{\frac{\Sigma_S}{\Sigma_T}}_{\infty}^2}\mathcal{M},\quad \mbox{with }V_0=\E\left[\sum_{m\in V}\omega_m \norm{\sqrt{\Sigma_T}\left(\tphi_m^{(0)}-\phi_m^{(0)}\right)}_{2}^2 \right],$
and hence are uniformly bounded w.r.t. $n$.
\end{theorem}

Let us now discuss the variance. Assuming that $K_{km}=O(1)$ and $\omega_m\sim 1/M$, one can deduce that $\mathcal{M}$ is of the order of $G * |M/G|^2 * 1/M^2$, which simplifies to $O(1/G)$. If, moreover, $\tpsi_k^{(n)}$ depends smoothly on $\Omega_k$, then the variance can be further reduced. This reduction is due to the fact that
$\displaystyle\tpsi_k^{(n)}-\tpsi_{k'}^{(n)}\sim |\Omega_{k'}-\Omega_k|$. This implies that the variance is small when using a relatively large $G$. However, in practice, choosing 
$G=1$ is often a better option because the implementation is significantly simpler. Though $\mathcal{M}=O(1)$ for $G=1$, one may use parallel computation or the time average of RSI (a consequence of the following theorem) to improve the accuracy.


\begin{theorem}(ergodicity)\label{convergenceoflow}
Suppose there are two positive constants, $0<\alpha\le \beta<\infty$ such that the discrete transition kernel $K_{km}$ satisfies $\alpha\le K_{km}\le \beta$, for $\forall k, m\in V$. Assume that 
\[
\frac{\beta}{\alpha}\norm{\frac{\Sigma_S}{\Sigma_T}}_{\infty}<1.
\]
Then, the following conditions hold:
\begin{enumerate}
\item There exists a unique invariant measure $\pi \in \mathcal{P}(\mathcal{S})$ for the RSI method. The marginal distribution $\pi_{\psi}$ in $S_{\psi}$ has a finite first moment, and this measure has the property that
\begin{equation}\label{eq:meanandvarofpsimarginal}
\E_{\pi_{\psi}}(\tpsi_m)=\E_{\pi}(\tpsi_m)=\psi_m, \quad \mathrm{Var}_{\pi_{\psi}}(\tpsi)
\le \frac{\mathcal{M}}{1-\norm{\frac{\Sigma_S}{\Sigma_T}}}_{\infty}.
\end{equation}
Here $(\psi_m)_{m\in V}$ is the exact solution of the DOM, $\E_{\pi_{\psi}}$ and $\mathrm{Var}_{\pi_{\psi}}$ represent the mean and variance under the marginal $\pi_{\psi}=\int_{\mathcal{S}_V}\pi(dv,\cdot)$ of the invariant measure $\pi$ respectively. 

\item  Let $\mu^n$ be the law of the states at the $n$th iteration for $(V^{(n)}, \bm{\tilde{\psi}})$ in $\mathcal{S}$. Then, there exists a constant $C>0$
and $\rho\in (0, 1)$ such that
\begin{equation}\label{eq:convpsimarginal}
W_1(\mu^n_{\psi}, \pi_{\psi}) \le C\rho^n.
\end{equation}
Moreover, the average of the states generated by the RSI method converges to the exact solution of the DOM. In particular, as $N\to\infty$,
\begin{equation}
\frac{1}{N}\sum_{n=1}^N \bm{\tpsi}^{(n)} \to \bm{\psi}:=(\psi_1, \cdots, \psi_M)
\end{equation}
almost surely in $\mathcal{S}_{\psi}$ under the metric $\mathrm{dist}(\cdot, \cdot)$.
\end{enumerate}
\end{theorem}

 \section{Proof of the Theorems}\label{sec:proof}
\subsection{Notations and apriori estimate}
    We introduce the following notations. For any $\br\in \D$ and given $\Omega_m$, let $\br=\br_{m}^{0}+s\Omega_m$, where $\br_{m}^{0}\in \pt\D$, $s\in \R^+$.
    Define $\br_m(s')=\br_{m}^{0}+s'\Omega_m$. Then \eqref{3deqpsisi} can be rewritten for $0< s'<s$ as
    \eq$
    \frac{\dd}{\dd s'}\psi_m^{(n)}(\br_m(s'))+\Sigma_T(\br_m(s'))\psi_m^{(n)}(\br_m(s'))=\Sigma_S(\br_m(s'))\phi_m^{(n-1)}(\br_m(s'))+Q(\br_m(s')).
    $
    The solution of \eqref{3deqpsisi}-\eqref{3dbcsi} can then be expressed by
  \begin{equation}\label{eq:solutionpsi}
    \begin{aligned}
        \psi_m^{(n)}(\br)=&e^{-\int_0^s \Sigma_T(\br_m(s')))\dd s'}
    \psi_{\pt\D}^-(\br_{m}^{0},\Omega_m)\\
    &+\int_0^s\left[\Sigma_S(\br_{m}(s'))\phi_m^{(n-1)}(\br_{m}(s'))+Q(\br_{m}(s'))\right]e^{-\int_{s'}^{s}\Sigma_T(\br_{m}(s''))\dd s^{\prime \prime}}\dd s'.
    \end{aligned}
 \end{equation}
   
We show that the solution of transport sweep can be stably solved. 
    \begin{lemma}\label{aolemmaA1}
		For each $m\in V=\{1,2,\cdots,M\}$, if $\psi_m(\br)$ is the solution of the system 
\begin{subequations}	\label{3deq}	
  \begin{numcases}{}
			&$\Omega_m \cdot \nabla \psi_m(\br)+\Sigma_T(\br)\psi_m(\br)=\Sigma_S(\br)\phi_m(\br)+Q(\br)$, \label{3deqpsi0}\\
			&$\psi_m(\br)=\psi_{\Gamma}^{-}(\br),\quad \br \in \pt \D, \quad \Omega_m\cdot \bn_{\br}<0$, \label{3dbc01} 
		\end{numcases}
  \end{subequations}
then $\psi_m(\br)$ satisfies the following a priori estimate:  
	\eq$\label{1dape0}\norm{\sqrt{\Sigma_T}\psi_m}_2\leq \norm{\frac{\Sigma_S}{\Sigma_T}}_{\infty} \norm{\sqrt{\Sigma_T}\phi_m}_2
 +\norm{\frac{Q}{\sqrt{\Sigma_T}}}_2+\frac{1}{\sqrt{2}}\norm{\psi_{\Gamma}^{-}(\br)}_{L^2(\pt\D_m^{-})},$
where $\pt\D_m^{-}=\{\br\in \pt\D: \Omega_m\cdot \bn_{\br}<0\}$.		
\end{lemma}
\begin{proof}

By the linearity of the equation, we can decompose the solution into three parts
\[
\psi_m=\psi_m^{1\ast}+\psi_m^{2\ast}+\psi_m^{3\ast},
\]
where $\psi_m^{1\ast}$ is the solution to \eqref{3deq} with $Q(\br)=0$ and $\psi_{\Gamma}^{-}(\br)=0$, $\psi_m^{2\ast}$ is the solution to \eqref{3deq} with $\phi_m(\br)=0$ and $\psi_{\Gamma}^{-}(\br)=0$, while $\psi_m^{3\ast}$ is the solution to \eqref{3deq} with both $Q(\br)=0$, $\phi_m(\br)=0$ and with nonzero boundary term.
Then,
\[
\norm{\sqrt{\Sigma_T}\psi_m}_2\leq \sum_{j=1}^3 \norm{\sqrt{\Sigma_T}\psi_m^{j\ast}}_2.
\]

For $\psi_m^{1\ast}$, multiplying both sides of \eqref{3deqpsi0} by $\psi^{1\ast}_m(\br)$ and integrating on $\D$, we obtain 
\begin{multline}
\int_{\D}\Sigma_S(\br)\phi_m(\br)\psi_m^{1\ast}(\br)\dd \br-\int_{\D}\Sigma_T(\br)(\psi_m^{1\ast})^2(\br)\dd \br=\frac{1}{2}\int_{\D}\Omega_m \cdot \nabla (\psi_m^{1\ast})^2(\br)\dd \br\\
= \frac{1}{2}\int_{\pt \D} \left(\Omega_m \cdot \bn_{\br}\right) (\psi_m^{1\ast})^2(\br) \dd S_{\br} \geq 0.
\end{multline}
The $L^2$ norm of $\sqrt{\Sigma_T}\psi_m^{1\ast}$ is controlled by 
		$$
\begin{aligned}
\norm{\sqrt{\Sigma_T}(\psi_m^{1\ast})}_2^2
    \leq \int_{\D}\Sigma_S(\br)\phi_m(\br)\psi_m^{1\ast}(\br)\dd \br
    \leq \norm{\frac{\Sigma_S}{\Sigma_T}}_{\infty}\norm{\sqrt{\Sigma_T}\psi_m^{1\ast}}_2\norm{\sqrt{\Sigma_T}\phi_m}_2,
\end{aligned}$$
The estimate for $\psi_m^{2\ast}$ is similar and we omit the details.

For $\psi_m^{3\ast}$, similar computation gives
\begin{multline}
\int_{\D}\Sigma_T(\br)(\psi_m^{3\ast})^2(\br)\dd \br= -\frac{1}{2}\int_{\pt \D^+} \left(\Omega_m \cdot \bn_{\br}\right) (\psi_m^{3\ast})^2(\br) \dd S_{\br}\\
-\frac{1}{2}\int_{\pt \D^-} \left(\Omega_m \cdot \bn_{\br}\right) (\psi_{\Gamma}^-)^2(\br) \dd S_{\br} \le \frac{1}{2}\int_{\pt \D^-}  (\psi_{\Gamma}^-)^2(\br) \dd S_{\br}.
\end{multline}
Then \eqref{1dape0} follows.
\end{proof}

Now, we move to RSI. We introduce filtration to clarify the dependence of $\tpsi_m^{(n)}$, $\tphi_m^{(n)}$ on $V^{(n)}$. Define the filtration $\{\F^{(n)}\}_{n\geq 0}$ by
	\eq$\F^{(n)}=\sigma\left(V^{(n')},0\le n'\leq n\right).$
	Then $\tq_{k}^{(n)}$ in \eqref{deftq}, $\tphi_m^{(n-1)}, \tpsi_m^{(n)}\in \F^{(n-1)}$.
 We emphasize that for given $\tpsi_k^{(n-1)}$ with $k\in V^{(n-1)}$ $\tpsi_m^{(n)}$ is defined for all $m\in V$, though in practical simulations only those in $V^{(n)}$ would be used.	

The following contraction result will be useful when studying the ergodicity properties.
\begin{lemma}[Contraction]\label{lmm:contract}
Suppose $\{\tpsi_{m,1}^{(n)}\}_{m\in V}$ and $\{\tpsi_{m,2}^{(n)}\}_{m\in V}$ are two realization of RSI for solving the same system \eqref{3ddom} with different $\{\psi_m^{(0)}\}_{m\in V}$, and they share the same $V^{(n)}$ in each iteration, then we have
		\eq$\sum_{m\in V}\omega_m\E\left[ \norm{\sqrt{\Sigma_T}\left(\tpsi_{m,1}^{(n+1)}-\tpsi_{m,2}^{(n+1)}\right)}_2\right]\leq \norm{\frac{\Sigma_S}{\Sigma_T}}_{\infty} \sum_{m\in V}\omega_m\E\left[\norm{\sqrt{\Sigma_T}\left(\tpsi_{m,1}^{(n)}-\tpsi_{m,2}^{(n)}\right)}_2\right].\label{eqcontraction}$
\end{lemma}
	\begin{proof}

		Note that $\tpsi_{m,1}^{(n+1)}-\tpsi_{m,2}^{(n+1)}$ is the solution to system \eqref{3deq} with the source term being $\tphi_{m,1}^{(n)}-\tphi_{m,2}^{(n)}$, and with $Q(\br) = 0$ and $\psi_{\Gamma}^{-}(\br) = 0$. Therefore, by Lemma \ref{aolemmaA1}, we obtain
\[ 
  \begin{split}
\Delta_m^{n+1}:= \norm{\sqrt{\Sigma_T}\left(\tpsi_{m,1}^{(n+1)}-\tpsi_{m,2}^{(n+1)}\right)}_2 
  & \leq \norm{\frac{\Sigma_S}{\Sigma_T}}_{\infty} \norm{\sqrt{\Sigma_T} \left(\tphi_{m,1}^{(n)}-\tphi_{m,2}^{(n)}\right)}_2 \\
  &=\norm{\frac{\Sigma_S}{\Sigma_T}}_{\infty} \norm{\sqrt{\Sigma_T}\sum_{k\in V^{(n)}}\omega_k K_{km} \tq_k^{(n)} \left(\tpsi_{k,1}^{(n)}-\tpsi_{k,2}^{(n)}\right)}_2 
  \end{split}
  \]
Take the expectation on both sides, we have
\[
\begin{aligned}
\E \Delta_m^{n+1}  
			\leq &  \norm{\frac{\Sigma_S}{\Sigma_T}}_{\infty}\E\left[\sum_{k\in V^{(n)}}\omega_k K_{km} \tq_k^{(n)} \norm{\sqrt{\Sigma_T}\left(\tpsi_{k,1}^{(n)}-\tpsi_{k,2}^{(n)}\right)}_2\right] \\
			= & \norm{\frac{\Sigma_S}{\Sigma_T}}_{\infty}\sum_{k\in V}\omega_k K_{km} \E\left[\norm{\sqrt{\Sigma_T}\left(\tpsi_{k,1}^{(n)}-\tpsi_{k,2}^{(n)}\right)}_2\E\left[\tq_k^{(n)}\I_{k\in V^{(n)}}\bigg|\F^{(n-1)}\right]\right] \\
			= &\norm{\frac{\Sigma_S}{\Sigma_T}}_{\infty}\sum_{k\in V}\omega_k K_{km} \E\left[\norm{\sqrt{\Sigma_T}\left(\tpsi_{k,1}^{(n)}-\tpsi_{k,2}^{(n)}\right)}_2 \right].
\end{aligned}
\]
Here, the first inequality follows from the triangle inequality.
Summing over all $m \in V$ and weighting by $\omega_m$, we obtain
\[ 
\begin{aligned}
\sum_{m\in V}\omega_m\E \Delta_m^{n+1}
& \leq  \sum_{m\in V}\omega_m \norm{\frac{\Sigma_S}{\Sigma_T}}_{\infty}\sum_{k\in V}\omega_k K_{km} \E\left[\norm{\sqrt{\Sigma_T}\left(\tpsi_{k,1}^{(n)}-\tpsi_{k,2}^{(n)}\right)}_2 \right] \\
&=  \norm{\frac{\Sigma_S}{\Sigma_T}}_{\infty} \sum_{k\in V}\left(\sum_{m\in V}\omega_m K_{km}\right)\omega_k \E\left[\norm{\sqrt{\Sigma_T}\left(\tpsi_{k,1}^{(n)}-\tpsi_{k,2}^{(n)}\right)}_2 \right].\\
\end{aligned}
\]
Thus, using \eqref{assK}, equation \eqref{eqcontraction} follows immediately.
\end{proof}

	\subsection{Proof of Theorem \ref{expectation}}\label{sec:prooftheorem}
	\begin{proof}
		We prove the theorem by induction. In the first step,  $\tpsi_m^{(0)}=\psi_m^{(0)}$ for $m\in V^{(0)}$. Then, for a given $k$, because $V^{(0)}$ is chosen randomly with uniform probability, the probability that the $k$th direction in group $V(k)$ is selected is $1/|V(k)|$. Thus, $$|V(k)|  \E\left[\tpsi_{k}^{(0)}\I_{k\in V^{(0)}}\right]=\tpsi_{k}^{(0)}=\psi_{k}^{(0)}(\br),$$ where $\I_{k\in V^{(0)}}$ is the characteristic function such that when $k\in V^{(0)}$, it is $1$, otherwise, it is $0$. From \eqref{tildephi0}, we have 
  \eq$\E\left[\tphi_m^{(0)}\right]= \sum_{k\in V}\omega_k K_{km}|V(k)|  \E\left[\tpsi_{k}^{(0)}\I_{k\in V^{(0)}}\right]=\sum_{k\in V}\omega_k K_{km}\psi_{k}^{(0)}(\br)=\phi_m^{(0)}.$  
Consequently, since the source $Q(\br)$ and boundary data are deterministic, by \eqref{eq:solutionpsi}, one has
\begin{equation}\label{eq:unbiasfirststep}
\E\left[\tpsi_m^{(1)}\right]=
\psi_m^{(1)}.
\end{equation}
  
    Next, suppose we have proved $ \E\left[\tphi_m^{(n-1)}\right]=\phi_m^{(n-1)}$, $\E\left[\tpsi_m^{(n)}\right]=\psi_m^{(n)}$, we are going to prove $\E\left[\tphi_m^{(n)}\right]=\phi_m^{(n)}$, $\E\left[\tpsi_m^{(n+1)}\right]=\psi_m^{(n+1)}$.
	Due to the form of $\tq_{k}^{(n)}$ in \eqref{deftq} and the way to choose $V^{(n)}$ according to \eqref{qkn}, we have $p_k^{(n)}=\E\left[\I_{k\in V^{(n)}}\big|\F^{(n-1)}\right]$. Then, from \eqref{1drsi} and 
 $\tq_{k}^{(n)}, \tpsi_{k}^{(n)}\in \F^{(n-1)}$, 
		\begin{align*}
		    \E\left[\tphi_m^{(n)}\right]
			=&\sum_{k\in V}\omega_k K_{km} \E\left[\tq_{k}^{(n)}\tpsi_{k}^{(n)}\I_{k\in V^{(n)}}\right]\\
			=&\sum_{k\in V}\omega_k K_{km} \E\left[\tq_{k}^{(n)}\tpsi_{k}^{(n)}\E\left[\I_{k\in V^{(n)}}\bigg|\F^{(n-1)}\right]\right]\\
			=&\sum_{k\in V}\omega_k K_{km} \E\left[\tpsi_{k}^{(n)}\right]=\sum_{k\in V}\omega_k K_{km} \psi_k^{(n)}=\phi_m^{(n)}.
		\end{align*}
		For the density flux $\tpsi_m^{(n)}$, by the same reasoning as in \eqref{eq:unbiasfirststep}, one has	$\E\left[\tpsi_m^{(n+1)}\right]
  =\psi_m^{(n+1)}$.   Hence Theorem \ref{expectation} has been proved for all $n$.
	\end{proof}

 \subsection{Proof of Theorem \ref{variance}}
	
	\begin{proof}[Proof of Theorem \ref{variance}]
	
By Theorem \ref{expectation},  it holds that
\[
\E\left[\left(\tphi_m^{(n)}-\phi_m^{(n)}\right)^2\right]
=\E\left[\left(\tphi_m^{(n)}\right)^2\right]-\left(\phi_m^{(n)}\right)^2.
\]
We first compute the second moment:
\begin{equation}
\E\left[\left(\tphi_m^{(n)}\right)^2\right] 
=\E\left[\sum_{k\in V}\sum_{k'\in V}\omega_k K_{km}\tq_{k}^{(n)} \tpsi_{k}^{(n)}\omega_{k'} K_{k'm}\tq_{k’}^{(n)} \tpsi_{k'}^{(n)}\E\left[\I_{k,k'\in V^{(n)}}\bigg|\F^{(n-1)}\right]\right].
\end{equation}
Note that 
\eq$\E\left[\I_{k,k'\in V^{(n)}}\bigg|\F^{(n-1)}\right]=\left\{
		\begin{aligned}
		&p_k^{(n)}, &&k'=k, \\
		&p_k^{(n)}p_{k'}^{(n)}, &&k'\notin V(k),  \\
		&0, &&k'\neq k, k'\in V(k).
		\end{aligned}
\right.$
Then, one has
\begin{multline}\label{Ephin}
\E\left[\left(\tphi_m^{(n)}\right)^2\right]
=\E\left[\sum_{k\in V}\tq_{k}^{(n)}\left(\omega_k K_{km}\tpsi_k^{(n)}\right)^2-\sum_{g=1}^G\left(\sum_{k\in V_g}\omega_k K_{km}\tpsi_k^{(n)}\right)^2\right]\\
+\E\left[\left(\sum_{g=1}^G\sum_{k\in V_g}\omega_k K_{km}\tpsi_k^{(n)}\right)^2\right]=:I_1+I_2.
\end{multline}
The $I_1$ term in \eqref{Ephin} can be estimated directly that
  $$\begin{aligned}
 I_1=&\E\left[\sum_{g=1}^G\left[\left(\sum_{k\in V_g}\tq_{k}^{(n)}\left(\omega_k K_{km}\tpsi_k^{(n)}\right)^2\right)\left(\sum_{k\in V_g}p_k^{(n)}\right)-\left(\sum_{k\in V_g}\omega_k K_{km}\tpsi_k^{(n)}\right)^2\right]\right]\\
	&\leq \E\left[\frac{1}{2}\sum_{g=1}^G\sum_{k\in V_g}\sum_{k'\in V_g}\left(\sqrt{\frac{p_{k'}^{(n)}}{p_{k}^{(n)}}}\omega_k K_{km}\tpsi_k^{(n)}-\sqrt{\frac{p_{k}^{(n)}}{p_{k'}^{(n)}}}\omega_{k'} K_{k'm}\tpsi_{k'}^{(n)}\right)^2\right]=:R.\\
	\end{aligned}$$
The second to last inequality is derived from algebraic calculations and the H\"older inequality. 
 
On the other hand, a direct estimation yields
\begin{equation*}
I_2-\left(\phi_m^{(n)}\right)^2 
		=\E\left[\left(\sum_{k\in V}\omega_k K_{km}\left(\tpsi_k^{(n)}-\psi_k^{(n)}\right)\right)^2\right] \leq\E\left[\sum_{k\in V}\omega_k K_{km}\left(\tpsi_k^{(n)}-\psi_k^{(n)}\right)^2\right],
\end{equation*}
 where the first equality utilizes $\E\left[\sum_{g=1}^G\sum_{k\in V_g}\omega_k K_{km}\tpsi_k^{(n)}\right]=\E\left[\sum_{k\in V}\omega_k K_{km}\tpsi_k^{(n)}\right]=\phi_k^{(n)}$ and the last inequality utilize $\sum_{k\in V}\omega_k K_{km}=1$ in \eqref{assK}.

Hence, we have that
\begin{gather*}
\E\left[\left(\tphi_m^{(n)}-\phi_m^{(n)}\right)^2\right]\le 
\E\left[\sum_{k\in V}\omega_k K_{km}\left(\tpsi_k^{(n)}-\psi_k^{(n)}\right)^2\right]+R.
\end{gather*}
Multiplying both sides of the inequality by $\Sigma_T$ and integrating with respect to $\br \in \D$, and using the definition of $\mathcal{M}$ below equation \eqref{assM}, we obtain
 \begin{equation}\label{eq:variencebound}
 \E \norm{\sqrt{\Sigma_T}\left(\tphi_m^{(n)}-\phi_m^{(n)}\right)}_{2}^2 \leq  \sum_{k\in V}\omega_k K_{km}\norm{\sqrt{\Sigma_T} \left(\tpsi_k^{(n)}-\psi_k^{(n)}\right)}_2^2+\mathcal{M}.
\end{equation}

Since $\tpsi_k^{(n)}-\psi_k^{(n)}$ solves system \eqref{3deq} with the source term $\tphi_k^{(n-1)}-\phi_k^{(n-1)}$, by Lemma \ref{aolemmaA1}, multiplying \eqref{eq:variencebound} by $\omega_m$ and summing over $m$ gives us  (noting $\sum_m \omega_m=1$, $\sum_{m\in V}\omega_{m}K_{km}=1$)

\begin{equation*}
\E\left[\sum_{m\in V}\omega_m \norm{\sqrt{\Sigma_T}\left(\tphi_m^{(n)}-\phi_m^{(n)}\right)}_{2}^2 \right] 
	    \leq\norm{\frac{\Sigma_S}{\Sigma_T}}_{\infty}^2\E\left[\sum_{k\in V}\omega_k \norm{\sqrt{\Sigma_T}\left(\tphi_m^{(n-1)}-\phi_m^{(n-1)}\right)}_{2}^2 \right]+\mathcal{M}.
\end{equation*}
By iterating this process, we achieve the desired variance control for the sequence ${\tphi_m^{(n)}}$.


    Then for $\tpsi_m^{(n)}-\psi_m^{(n)}$, by Lemma \ref{aolemmaA1}, one has
    $$
    \begin{aligned}
        \E\left[\sum_{m\in V}\omega_m\int_{\D}\Sigma_T\left(\tpsi_m^{(n)}-\psi_m^{(n)}\right)^2\dd \br\right]\leq & \norm{\frac{\Sigma_S}{\Sigma_T}}_{\infty}^2 \E\left[\sum_{m\in V}\omega_m\int_{\D}\Sigma_T\left(\tphi_m^{(n-1)}-\phi_m^{(n-1)}\right)^2\dd \br\right]\\
        \leq&\norm{\frac{\Sigma_S}{\Sigma_T}}_{\infty}^{2n}V_0+\frac{\norm{\frac{\Sigma_S}{\Sigma_T}}_{\infty}^2-\norm{\frac{\Sigma_S}{\Sigma_T}}_{\infty}^{2n}}{1-\norm{\frac{\Sigma_S}{\Sigma_T}}_{\infty}^2}\mathcal{M}.
    \end{aligned}$$
Using the fact that $\norm{\frac{\Sigma_S}{\Sigma_T}}_{\infty}^2\le 1$ gives the result.
	\end{proof}



\subsection{Proof of Theorem \ref{convergenceoflow}}

In Theorem \ref{convergenceoflow}, we view the iterates of RSI method as the states of a Markov chain.
Our goal is to investigate the convergence of RSI to the solution of the DOM system in \eqref{3deqphi} using the tool of Markov chains. In particular, we study the convergence rate of the laws (distributions) of the states to the invariant measure in RSI. This would give the accurate description of the  performance of the algorithm.

The following lemma gives the boundedness of the solutions of RSI.
\begin{lemma}\label{lmm:boundness}
Suppose  $\alpha\le K_{km}\le \beta, \quad \forall k, m\in V$ and
$\lambda:=\frac{\beta}{\alpha}\norm{\frac{\Sigma_S}{\Sigma_T}}_{\infty}<1$.
Then, there is a constant $C>0$ which is independent of $\{V^{(m)}\}_{m\ge 0}$ such that almost surely
\begin{gather}\label{eq:psimuniformbound}
\sup_n \sup_{m\in V}  \norm{\sqrt{\Sigma_T} \tpsi_m^{(n)}}_2\le C.
\end{gather}
\end{lemma}
\begin{proof}
By the definition in \eqref{1drsi} and Lemma \ref{aolemmaA1}, we find that
\[
\norm{\sqrt{\Sigma_T}\tilde{\phi}_m^{(n)}}_2
 \le  \sum_{k\in V^{(n)}}\omega_k K_{km}\tilde{q}_k^{(n)}
  \left(\norm{\frac{\Sigma_S}{\Sigma_T}}_{\infty} \norm{\sqrt{\Sigma_T}\tilde{\phi}_k^{(n-1)}}_2+C\right),
 \]
 where $C$ is an upper bound for the terms related to the source, boundary terms and $\Omega_m$, and thus independent of $n$ and $V^{(n)}$.

 By the definition,
\[
\begin{split}
 \sum_{k\in V^{(n)}}\omega_k K_{km}\tilde{q}_k^{(n)}
 &=\sum_{k\in V^{(n)}} K_{km}\frac{\sum_{m'\in V(k)}\omega_{m'}c_{m'}^{(n-1)}}{c_k^{(n-1)}}\\
 &=\sum_{k\in V^{(n)}}\frac{K_{km} \sum_{k'\in V^{(n-1)}}\omega_{k'} \sum_{m'\in V(k)}\omega_{m'}K_{k'm'}}{\sum_{k'\in V^{(n-1)}} \omega_{k'}K_{k' k} }.
 \end{split}
 \]
 Using the lower and upper bounds of $K_{km}$, one then has
 \begin{multline*}
\sum_{k\in V^{(n)}}\frac{K_{km} \sum_{k'\in V^{(n-1)}}\omega_{k'} \sum_{m'\in V(k)}\omega_{m'}K_{k'm'}}{\sum_{k'\in V^{(n-1)}} \omega_{k'}K_{k' k} } \\
\le \frac{\beta}{\alpha}
\sum_{k\in V^{(n)}}\frac{\sum_{k'\in V^{(n-1)}}\omega_{k'} \sum_{m'\in V(k)}\omega_{m'}K_{k'm'}}{\sum_{k'\in V^{(n-1)}} \omega_{k'} }
=\frac{\beta}{\alpha}.
 \end{multline*}
 With the assumption on $\lambda$, one then has
 \[
 \sup_{m\in V^{(n+1)}}\norm{\sqrt{\Sigma_T}\tilde{\phi}_m^{(n)}}_2\le  \lambda \sup_{k\in V^{(n)}} \norm{\sqrt{\Sigma_T}\tilde{\phi}_k^{(n-1)}}_2
 +\frac{\beta}{\alpha}C.
 \]
 By induction, it is easy to verify that $\sup_{m\in V^{(n+1)}} \norm{\sqrt{\Sigma_T}\tilde{\phi}_m^{(n)}}_2$
 is uniformly bounded with respect to $n$. Then,
 by Lemma \ref{aolemmaA1}, $\sup_{m\in V^{(n+1)}} \norm{\sqrt{\Sigma_T}\tilde{\psi}_m^{(n)}}_2$ is thus uniformly bounded with respect to $n$.
Using \eqref{1drsi}, we can define $ \tilde{\phi}_m^{(n)}$ and thus $ \tilde{\psi}_m^{(n)}$ for all $m\in V$, not just those in $V^{(n+1)}$, though only those in $V^{(n+1)}$ would be used. Since $\sum_{k\in V^{(n+1)}}\omega_k K_{km}\tilde{q}_k^{(n+1)}\le \beta/\alpha$, \eqref{eq:psimuniformbound} holds.
\end{proof}

By Lemma \ref{lmm:contract}, it is clear that if all the implementations start with the same $V^{(0)}$, then the convergence of the laws of $\bm{\tilde{\psi}}$ would directly result from the contraction property. However, the problem is that the initial $V^{(0)}$ could be different. Hence, besides the contraction property with the same $V^{(n)}$, we also need the ergodicity of the set $V^{(n)}$. 

As we have mentioned, $V^{(n)}$ is a Markov chain in $\mathcal{S}_V$. The state space $\mathcal{S}_V$ is finite, with $N = |\mathcal{S}_V| < \infty$.  To study the ergodicity of the chain, consider two initial random values $V_1^{(0)}$ and $V_2^{(0)}$ for this chain. Let the corresponding laws be $\mu_{V,1}^{n}$ and $\mu_{V,2}^n$ respectively. We consider the following classical coupling. 

Suppose that $P_{ij}=\mathbb{P}(v_j| v_i)$ is the transition probability where $v_i\in \mathcal{S}_V\subset \{0, 1\}^M$. Define $c_j:=\inf_i P_{ij}$ and consider
\[
c=\sum_{j=1}^N c_j>0.
\]
Let $V_1^{(n)}$ and $V_2^{(n)}$ be defined by the following.
\begin{itemize}
\item Construct a coupling for $V_1^{(0)}$ and $V_2^{(0)}$ (for instance, one can make them independent).
\item Generate a sequence of i.i.d. Bernoulli variables $\alpha_n$ such that $\mathbb{P}(\alpha_n = 1) = c$.
For each $n\ge 1$, we do the following:
\begin{enumerate}
\item If $V_1^{(n-1)}=V_2^{(n-1)}$, we sample $V_1^{(n)}$ and $V_2^{(n)}$ according to the transition probability $P_{ij}$.
\item If $V_1^{(n-1)}\neq V_2^{(n-1)}$, when $\alpha_n=1$, we set
\[
V_1^{(n)}=V_2^{(n)}=u_n,
\]
where $u_n$ is sampled from the measure $\nu$ defined by $\nu({j}) = c_j/c$, independent of all previous random variables. 
If $\alpha_n = 0$, then $V_1^{(n)}$ and $V_2^{(n)}$ are generated independently, with the transition probability
\[
\mathbb{P}(v_j|v_i)=\frac{P_{ij}-c\nu_j}{1-c}.
\]
\end{enumerate}
\end{itemize}

By the construction above, the chain $V_1^{(n)}$ will have law $\mu_{V,1}^n$ while $V_2^{(n)}$ will have law $\mu_{V,2}^n$. In other words, if we look at them separately, they are implementations of the Markov chain with the given initial values. However, in the coupling, as soon as $V_1^{(n)}$ and $V_2^{(n)}$
become the same at some point, they move together for later times. This then allows us to gauge the difference between the laws by estimating the probability that they are not equal. At each time step, if they are not equal at $n-1$, then there is at least a probability $c$ that they become the same at time $n$. 
Based on this, one can directly conclude the following fact.
\begin{lemma}\label{lmm:survivingevent}
Consider the event for the coupling constructed above
\begin{gather}
E_n:=\{\omega\in \bar\Omega: V_1^{(n)} \neq V_2^{(n)}\}.
\end{gather}
Then, it satisfies that
\begin{equation}
TV(\mu_{V,1}^n, \mu_{V,2}^n) \le \mathbb{P}(E_n)\le (1-c)^n,
\end{equation}
where $TV$ indicates the total variation norm.
\end{lemma}

The following is an observation for our RSI iteration.
\begin{lemma}
In the RSI method, the constant $c$ above has a lower bound given by
\begin{equation}
c\ge \theta^G, \quad \theta=\inf_n \min_{g=1,\cdots, G} \min_{m, m'\in V_g}\frac{c_m^{(n)}}{c_{m'}^{(n)}}.
\end{equation}
In particular, if the transition kernel $K_{km}$ has bounds given by $\alpha\le K_{km}\le \beta, \quad \forall k, m\in V$, then one can take
\begin{equation}
c\ge \left(\frac{\alpha}{\beta}\right)^G.
\end{equation}
\end{lemma}
The proof is a direct calculation: $p_m^{(n)}\ge \theta\frac{\omega_m}{\sum_{\omega'\in V(m)}\omega_{m'}}$. Then, for each possible set $v\in \{0, 1\}^{M}$ with $\sum_{i\in V_g} v_i=1$ for each $g=1,\cdots, G$, one has
$\mathbb{P}(v | v')\ge \theta^{G}\prod_{g=1}^G (\frac{\omega_{m_g}}{\sum_{\omega'\in V_g}\omega_{m'}})$, where $m_g$ is the index in $V_g$ such that $v_{m_g}=1$. Summing over all possible $v$ gives the desired bound.

\begin{remark}
The bounds $\alpha, \beta$ of $K$ should not depend on the number of ordinates (i.e. $M$ or $G$) and thus $\theta$ is expected to be $O(1)$. For the power $G$ in $c\ge \theta^G$, this may seem restrictive if $G$ is large. Inside each $V_g$, if $K_{km}$ changes slowly, the number $\theta$ could be close to $1$ and this may improve the estimate. Another way to improve the mixing rate is to consider some careful construction, which we leave for future work. Nevertheless, we often take $G=1$, and this is not a problem.
\end{remark}

Considering the RSI algorithm, it generates a homogeneous Markov chain 
in $\mathcal{S}=\mathcal{S}_V\times
\mathcal{S}_{\psi}$. 
We now use this observation and the results above to prove Theorem \ref{convergenceoflow}. The main goal is to establish the ergodicity of the chain. Then, by the Birkhoff ergodic theorem, the iteration average converges to the mean of the stationary distribution, which is the solution of equation \eqref{3ddom}.

\begin{proof}[Proof of Theorem \ref{convergenceoflow}]

We consider two chains $(V_1^{(n)}, \bm{\tilde{\psi}}_1^{(n)})$ and $(V_2^{(n)}, \bm{\tilde{\psi}}_2^{(n)})$ with some given initial values. To study the laws of the two chains, we let $V_1^{(n)}$ and $V_2^{(n)}$ be coupled as constructed above.  Consider the events
\[
E_m:=\{\omega\in \bar\Omega: V_1^{(m)} \neq V_2^{(m)}\},\quad F_m=E_m\setminus E_{m+1}.
\]
By Lemma \ref{lmm:survivingevent}, $\mathbb{P}(E_m)\le (1-c)^n$ so that $\mathbb{P}(\cup_m F_m)=1$.

Conditioning on $F_m$, for any $n\ge m$, one has  by Lemma \ref{lmm:contract} that
\begin{equation*}
\mathbb{E}(\mathrm{dist}(\bm{\tilde{\psi}}_1^{(n)},  \bm{\tilde{\psi}}_2^{(n)}) | F_m)
\le \norm{\frac{\Sigma_S}{\Sigma_T}}_{\infty}^{n-m}  \mathbb{E}(\mathrm{dist}(\bm{\tilde{\psi}}_1^{(m)},  \bm{\tilde{\psi}}_2^{(m)})| F_m).
\end{equation*}
By Lemma \ref{lmm:boundness}, one has almost surely that
\[
\Delta^m:=\mathrm{dist}(\bm{\tilde{\psi}}_1^{(m)},  \bm{\tilde{\psi}}_2^{(m)})\le C.
\]
Since $\mathbb{P}(F_m)\le \mathbb{P}(E_m)$, this bound and the estimate above indicate that 
\begin{equation}\label{eq:convergenceofpsimarginal}
\begin{split}
\mathbb{E}(\mathrm{dist}(\bm{\tilde{\psi}}_1^{(n)},  \bm{\tilde{\psi}}_2^{(n)}))
&\le \sum_{m=0}^{n-1}\mathbb{P}(F_m)\norm{\frac{\Sigma_S}{\Sigma_T}}_{\infty}^{n-m}  \mathbb{E}(\Delta^m | F_m)
+\mathbb{P}(E_{n})\mathbb{E}(\Delta^n | E_n)\\
&\le C\left(\sum_{m=0}^{n-1} \mathbb{P}(E_m)\norm{\frac{\Sigma_S}{\Sigma_T}}_{\infty}^{n-m}
+\mathbb{P}(E_{n})\right)\le C\rho^n,
\end{split}
\end{equation}
where $\rho$ can be taken to be $\max(1-c, \norm{\frac{\Sigma_S}{\Sigma_T}}_{\infty})+\epsilon$ for any $\epsilon>0$ such that $\rho<1$.

The estimate \eqref{eq:convergenceofpsimarginal} is the key estimate. To complete the proof, we also define a distance for $v\in \mathcal{S}_V$ by 
\[
\mathrm{dist}_V(v, v')=\sum_{m\in V} \omega_m |v_m-v_m'|.
\]
Clearly, this distance has an upper bound of $1$. The distance between two states $(v, \bm{\tilde{\psi}}_1)$
and $(v', \bm{\tilde{\psi}}_2)$ in $\mathcal{S}=\mathcal{S}_V\times\mathcal{S}_{\psi}$ is given by
\[
d((v, \bm{\tilde{\psi}}_1), (v', \bm{\tilde{\psi}}_2))
:=\mathrm{dist}_V(v, v')+\mathrm{dist}(\bm{\tilde{\psi}}_1,  \bm{\tilde{\psi}}_2).
\]
With the estimate \eqref{eq:convergenceofpsimarginal} and the fact that $V_1^{(n)}=V_2^{(n)}$ outside $E_n$, we find that the Wasserstein distance between the laws of the two chains satisfies
\[
W_1(\mu_1^n, \mu_2^n)\le \E(\mathrm{dist}_V(v, v')+\mathrm{dist}(\bm{\tilde{\psi}}_1,  \bm{\tilde{\psi}}_2))\le  \mathbb{P}(E_n)\cdot 1+C\rho^n
\le C'\rho^n.
\]
With this in hand, by a standard argument (see, for example, \cite{hairer2011yet}), one can conclude that the Markov chain has a unique invariant measure $\pi$. Indeed, we can take the second chain $(V_2^{(n)}, \bm{\psi}_2)$ to be the one starting from the second step of the first chain, which then implies that $\mu_2^{(n)}=\mu_1^{(n+1)}$. Then, $\{\mu_1^{(n)}\}$ forms a Cauchy sequence and thus has a limit $\pi$ in $W_1$. It is then straightforward to verify that this is the invariant measure and it is unique. As soon as we have this invariant measure, we can take the second chain $(V_2^{(n)}, \bm{\tilde{\psi}}_2)$  to be the one starting with the invariant measure. The estimate \eqref{eq:convergenceofpsimarginal} then gives \eqref{eq:convpsimarginal}.
The last claim about the iteration average is a simple application of the Birkhoff ergodic theorem (\cite{benaim2022markov}, Chapter 4.6.1).

Next, we justify the mean and variance of the marginal distribution of the invariant measure in $\mathcal{S}_{\psi}$.
Consider the iteration of RSI and let $\mu_n$ be the law in $\mathcal{S}$. Since it converges in Wasserstein-1 distance to $\pi$, then one has
\[
\int_{\mathcal{S}} \bm{\tpsi} d\mu_n
\to \int_{\mathcal{S}} \bm{\tpsi} d\pi.
\]
Note that $\E \bm{\tpsi}^{(n)}=\bm{\psi}^{(n)}$ and $\bm{\psi}^{(n)}$ is the iterates in the source iteration by Theorem \ref{expectation}. Since it converges to the solution of DOM, we thus conclude the first equation in \eqref{eq:meanandvarofpsimarginal}.  For the variance, according to Theorem \ref{variance}, we find that 
\[
\begin{split}
\E \left(\mathrm{dist}(\bm{\tpsi}^{(n)}, \bm{\psi}^{(n)})\right)^2
&=\E \left(\sum_{m=1}^M \omega_m \|\sqrt{\Sigma_T}(\tpsi_m^{(n)}-\psi_m^{(n)})\|_2\right)^2 \\
&\le \E\sum_{m=1}^M \omega_m \|\sqrt{\Sigma_T}(\tpsi_m^{(n)}-\psi_m^{(n)})\|_2^2 \\
&\le \norm{\frac{\Sigma_S}{\Sigma_T}}^{2n}V_0+\frac{\mathcal{M}}{1-\norm{\frac{\Sigma_S}{\Sigma_T}}_{\infty}^2}.
\end{split}
\]

Let $a_n=\mathrm{dist}(\bm{\tpsi}^{(n)}, \bm{\psi}^{(n)})$, $b_n=\mathrm{dist}(\bm{\tpsi}^{(n)}, \bm{\psi})$ and $\delta_n=\mathrm{dist}( \bm{\psi}^{(n)}, \bm{\psi})$, one then has
\[
a_n^2\ge (b_n-\delta_n)^2
=b_n^2-2b_n\delta_n+\delta_n^2
\ge b_n^2-\delta_n b_n^2-\delta_n+\delta_n^2.
\]
Since $\delta_n\to 0$ by the convergence of SI, we then find that
\[
\limsup_{n\to\infty} \E (\mathrm{dist}(\bm{\tpsi}^{(n)}, \bm{\psi}))^2
\le \frac{\mathcal{M}}{1-\|\frac{\Sigma_S}{\Sigma_T}\|_{\infty}^2}.
\]
Since $\mu_n$ converges to $\pi$ in $W_1$, it then converges weakly to $\pi$ and thus one has
\[
\E_{\pi} (\mathrm{dist}(\bm{\tpsi}, \bm{\psi}))^2
\le \liminf_{n\to\infty} \E (\mathrm{dist}(\bm{\tpsi}^{(n)}, \bm{\psi}))^2.
\]
The claim about variance then follows. 

\end{proof}

\section{Numerical results}\label{sec:numer}

In this section, we present the performance of the RSI for two benchmark tests that exhibit ray effects. We demonstrate the numerical convergence rates of RSI, validate the theoretical results, and showcase its effectiveness in mitigating the ray effect.

Considering the XY geometry, we use a computational domain defined as $[x_l,x_r]\times[y_b,y_t]=[0,1]\times[0,1]$. 
  Equations \eqref{3deqpsi} to \eqref{3deqphi} become the following form:  
 \begin{subequations}\label{2deq}
       \begin{numcases}{}
		&$\left(c_m \frac{\pt}{\pt x} +s_m \frac{\pt}{\pt y} +\Sigma_T(x,y)\right)\psi_m(x,y)=\Sigma_S(x,y)\phi_m(x,y)+Q(x,y)$, \label{2deqpsi}\\
		&$\psi_m(x_l,y)=\psi_l(y,c_m,s_m),\  c_m>0,\quad \psi_m(x_r,y)=\psi_r(y,c_m,s_m),\  c_m<0.$ \label{2dbc1} \\
        &$\psi_m(x,y_b)=\psi_b(x,c_m,s_m),\  s_m>0,\quad \psi_m(x,y_t)=\psi_t(x,c_m,s_m),\  s_m<0,$\label{2dbc2}  \\
        &$\displaystyle\phi_m(\br)=\sum_{k\in V} \omega_k K_{km}\psi_k(\br),  $
	\end{numcases} 
    \end{subequations}
where for $m\in V$,
\eq$c_m=\sqrt{1-\xi_m^2}\cos\theta_m,\quad s_m=\sqrt{1-\xi_m^2}\sin\theta_m,\quad \xi_m\in (0,1),\ \theta_m\in(0,2\pi).$

 For angular discretization, we employ uniform angular grids within the $(\xi,\theta)$ plane. Assuming the number of ordinates to be $M=N^2$,  the uniform grids in the rectangular domain $(\xi,\theta)\in(0,1)\times (0,2\pi)$ are given by:
\eq$
     \xi_l=\frac{2l-1}{2N}, \qquad\theta_k=\frac{2\pi(2k-1)}{2N},\quad l,k=1,2,\cdots,N.
$

Spatial discretization is performed using the classical Diamond Difference (DD) method. We use a uniform spatial mesh of size $I\times J$. For $i=0,\cdots, I$ and $j=0,\cdots, J$, the spatial grid points are defined as $x_i=\frac{i}{I}$, $y_j=\frac{j}{J}$. The midpoints of these grids, referred to as half-grids, are given by $x_{i+1/2}=(x_{i}+x_{i+1})/2$ and $y_{j+1/2}=(y_{j}+y_{j+1})/2$. Let $\psi_{m,i+1/2,j}\approx \psi_m(x_{i+1/2},y_j)$, $\psi_{m,i,j+1/2}\approx \psi_m(x_{i},y_{j+1/2})$, denote the approximate values of $\psi_m$ at the half-grid points. The DD spatial discretization of equation \eqref{2deq} is as follows \cite{LewisMiller}:
\begin{subequations}\label{2deqdc}
       \begin{numcases}{}
		&$c_m \frac{\psi_{m,i+1/2,j}-\psi_{m,i-1/2,j}}{x_{i+1/2}-x_{i-1/2}} +s_m \frac{\psi_{m,i,j+1/2}-\psi_{m,i,j-1/2}}{y_{j+1/2}-y_{j-1/2}} +\Sigma_{T,i,j}\psi_{m,i,j}=\Sigma_{S,i,j}\phi_{m,i,j}+Q_{i,j},$ \label{2deqpsidc}\\
		&$\psi_{m,i,j}=\frac{1
        }{2}\left(\psi_{m,i+1/2,j}+\psi_{m,i-1/2,j}\right)=\frac{1
        }{2}\left(\psi_{m,i,j+1/2}+\psi_{m,i,j-1/2}\right), $\\
        &$\psi_{m,1/2,j}=\psi_l(y_j,c_m,s_m),\  c_m>0,\quad \psi_{m,I+1/2,j}=\psi_r(y_j,c_m,s_m),\  c_m<0.$ \\
        &$\psi_{m,i,1/2}=\psi_b(x_i,c_m,s_m),\  s_m>0,\quad \psi_{m,i,J+1/2}=\psi_t(x_i,c_m,s_m),\  s_m<0,$ \\
        &$\displaystyle\phi_{m,i,j}=\sum_{k\in V} \omega_k K_{km}\psi_{k,i,j},  $
	\end{numcases} 
    \end{subequations}


Both isotropic and anisotropic scattering kernels are considered. For anisotropic scattering, we consider the following kernel:
\eq$\K(\Omega,\Omega')=\frac{1+c\cdot c'+s\cdot s'}{2\pi}.$ 
Unless otherwise specified, we use $G=1$, meaning that only one ordinate is selected in each iteration step. This selection minimizes the computational cost per sample and is optimal for parallel computation. The primary objective of this paper is to propose and establish the properties of RSI.   Discussions about parallel efficiency will be addressed in our future work, as spatial parallelization is also important.

Let $u_{i,j}\approx u(x_i,y_j)$. One can define the discrete $\ell^2$ norm of a function $u(x,y)$ as follows:
$$
\|u\|_{2}=\sqrt{\frac{1}{IJ}\sum_{i=1}^{I}\sum_{j=1}^{J}|u_{i,j}|^2}.
$$

\begin{figure}[!htbp]
	\centering
	\subfloat[]{\includegraphics[width=0.45\linewidth]{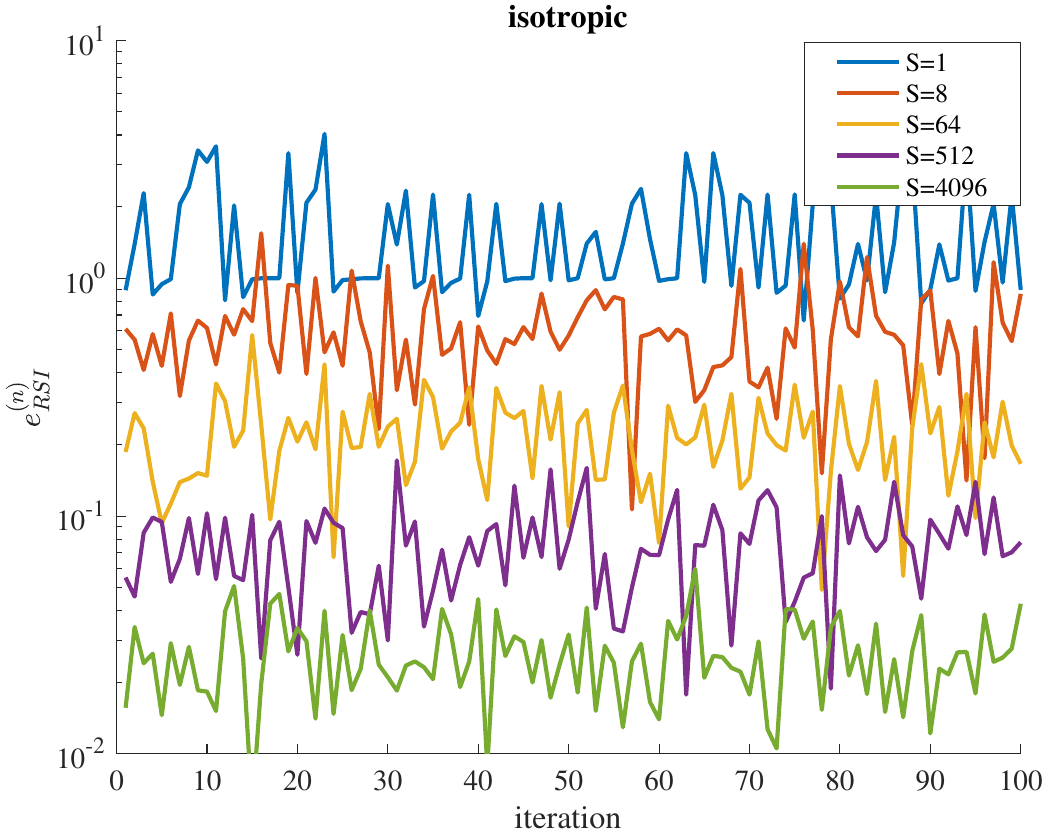}\label{errnisoex1}}
	\subfloat[]{\includegraphics[width=0.45\linewidth]{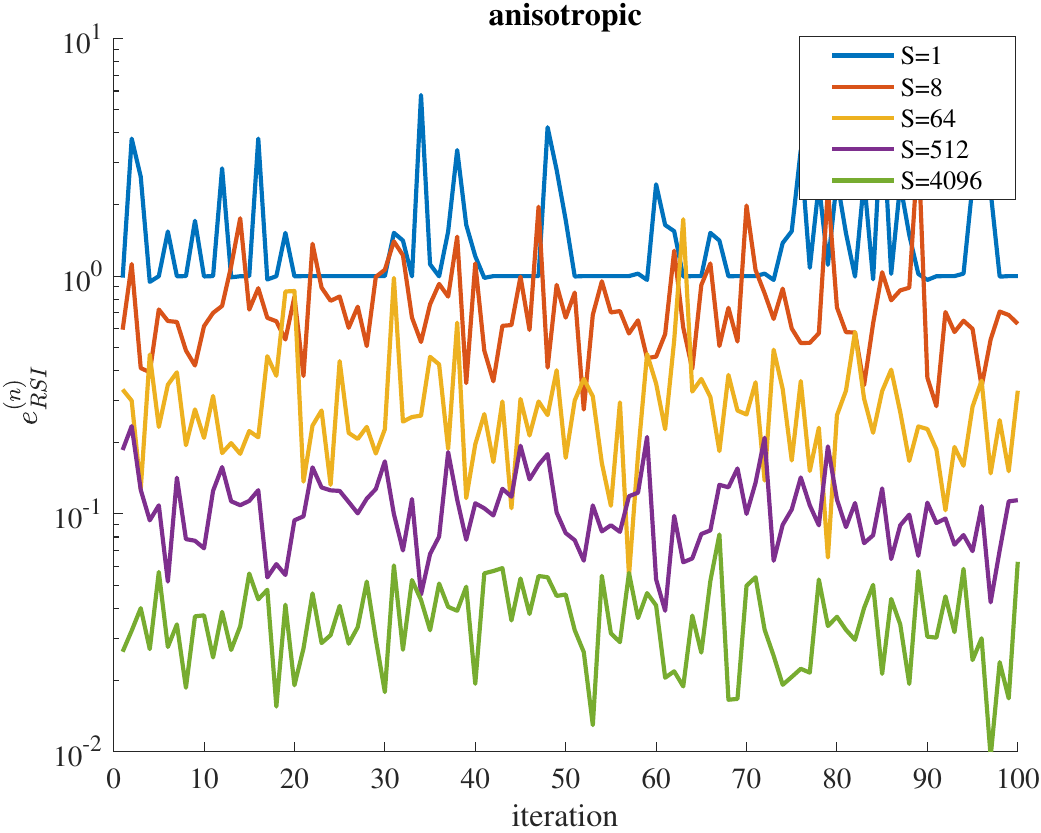}\label{errnanisoex1}}
	\caption{Example 1. The relative $\ell^2$ errors, $e_{RSI}^{(n)}$  defined as in \eqref{ersi} for different iteration steps $n$ and various numbers of samples $S$. (a) isotropic scattering; (b) anisotropic scattering. In this example $M=2^4$.}
	\label{2Derrnex1}
\end{figure}

\paragraph{Example 1.}The cross sections are constants, defined such that the total cross section $\Sigma_T=1$, and the absorption cross section $\Sigma_a=0.5$ with a zero source term.
 The inflow condition at the left boundary, for $s>0$, is given by:
$$\psi(x,0,c,s)=0, \mbox{ for }x\in [0,0.4]\cup[0.6,1], \quad  \psi(x,0,c,s)=10e^{-c^2-s^2}, \mbox{ for }x\in [0.4,0.6]. 
$$The inflow condition is a vacuum at all other boundaries. For spatial discretization in this example, we set $I=J=100$.

\paragraph{Convergence of RSI.} The particle density $\phi_{0,i,j}=\int_{S^2}\psi^{(n)}(x_i,y_j,\Omega)\dd \Omega$ in each iteration step is approximated by
\eq$\phi_{0,i,j}^{(n)}=\sum_{k\in V}\omega_k \psi^{(n)}_{k,i,j},\quad \tphi_{0,i,j}^{(n)}=\sum_{k\in V^{(n)}}\omega_k \tilde{q}_{k}^{(n)}\tpsi_{k,i,j}^{(n)}.\label{densityphi}$
in SI and RSI respectively. We assume that the SI method stops at the $N$th iteration step, such that \eq$\|\phi_{0}^{(N+1)}-\phi_{0}^{(N)}\|_2\Big/\|\phi_{0}^{(N))}\|_2<tol=10^{-10}.
\label{defN}$
 The relative $\ell^2$ error of the SI method for the $n$th iteration step is defined by
$$e_{SI}^{(n)}=\|\phi_{0}^{(n)}-\phi_{0}^{(N)}\|_2\Big/\|\phi_{0}^{(N))}\|_2.$$
 When $M=2^4$, the SI method stops at the $16$th and $7$th iteration steps, respectively, for isotropic and anisotropic scattering.

If we execute $S$ RSI samples, $\tphi^{(n)}_{0,i,j,s}$ represents the value of $\tphi^{(n)}_{0,i,j}$as given in equation \eqref{densityphi} for the $s$-th sample. The average of all $S$ samples is defined as
\eq$\bphi_{0,i,j}^{(n)}=\frac{1}{S}\sum_{s=1}^S\tphi_{0,i,j,s}^{(n)}.$
For each iteration step, we define the relative $\ell^2$ error of RSI by:
\eq$e_{RSI}^{(n)}=\|\bar\phi_{0}^{(n)}-\phi_{0}^{(n)}\|_2\Big/\|\phi_{0}^{(n))}\|_2,\label{ersi}$
where $\phi_{0}^{(n)}$ is the particle density obtained by SI in the $n-$th iteration step, as defined in equation \eqref{densityphi}.

In Figure \ref{2Derrnex1}, the values of $e_{RSI}^{(N)}$ are displayed for a given $M=2^4$ and various numbers of samples $S=1,8,64,512,4096$. Tests are conducted for both isotropic and anisotropic scattering cases. It can be observed that the expectation and variance of $e_{RSI}^{(N)}$ are independent of 
 $n$ and $e_{RSI}^{(N)}$ decreases as $S$ increases. 

 \begin{figure}[!htbp]
	\centering
	\subfloat[]{\includegraphics[width=0.45\linewidth]{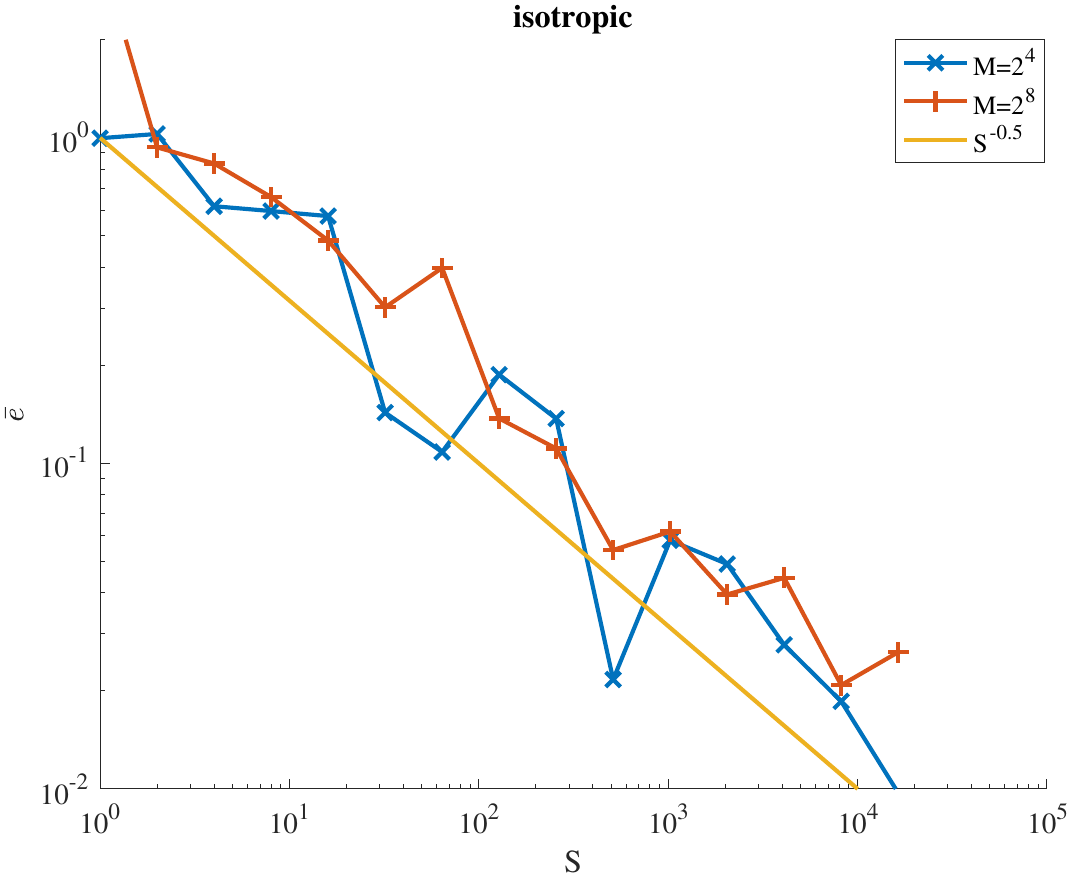}\label{errmSisoex1}}
	\subfloat[]{\includegraphics[width=0.45\linewidth]{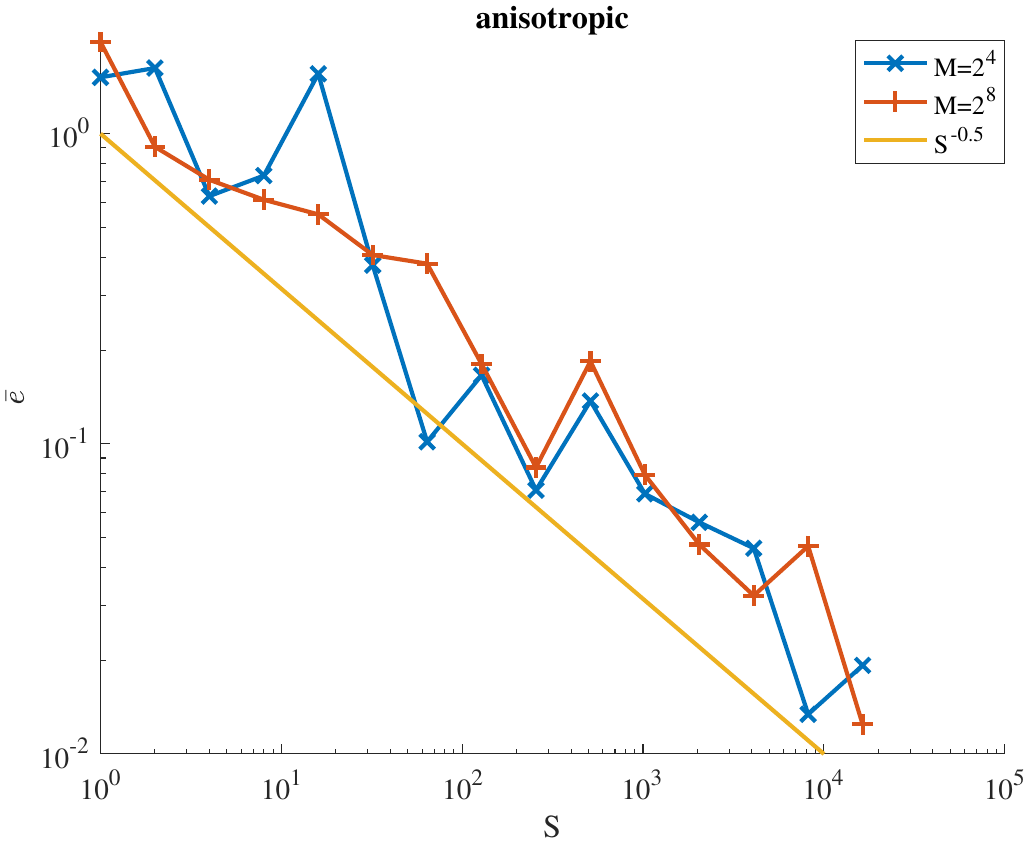}\label{errmSanisoex1}}
	\caption{Example 1. The convergence order of $e_{RSI}^{(N)}$ w.r.t the number of samples $S$, for $M=2^4$ and $2^8$. (a) isotropic scattering; (b) anisotropic scattering.}
	\label{2DisowrtS}
\end{figure}

To examine the convergence order of RSI with respect to the number of samples $S$, Figure \ref{2DisowrtS} presents the values of $e_{RSI}^{(N)}$ obtained for different $S$. Both isotropic and anisotropic scattering cases are considered, and results for different numbers of ordinates $M=2^4, 2^8$ are displayed. As observed from Figure \ref{2DisowrtS}, the convergence order of $e_{RSI}^{(N)}$ with respect to $S$ is $0.5$, which is independent of the number of ordinates $M$. This convergence order is identical to that of the Monte Carlo method.

\begin{figure}[!htbp]
	\centering
	\includegraphics[width=0.9\linewidth]{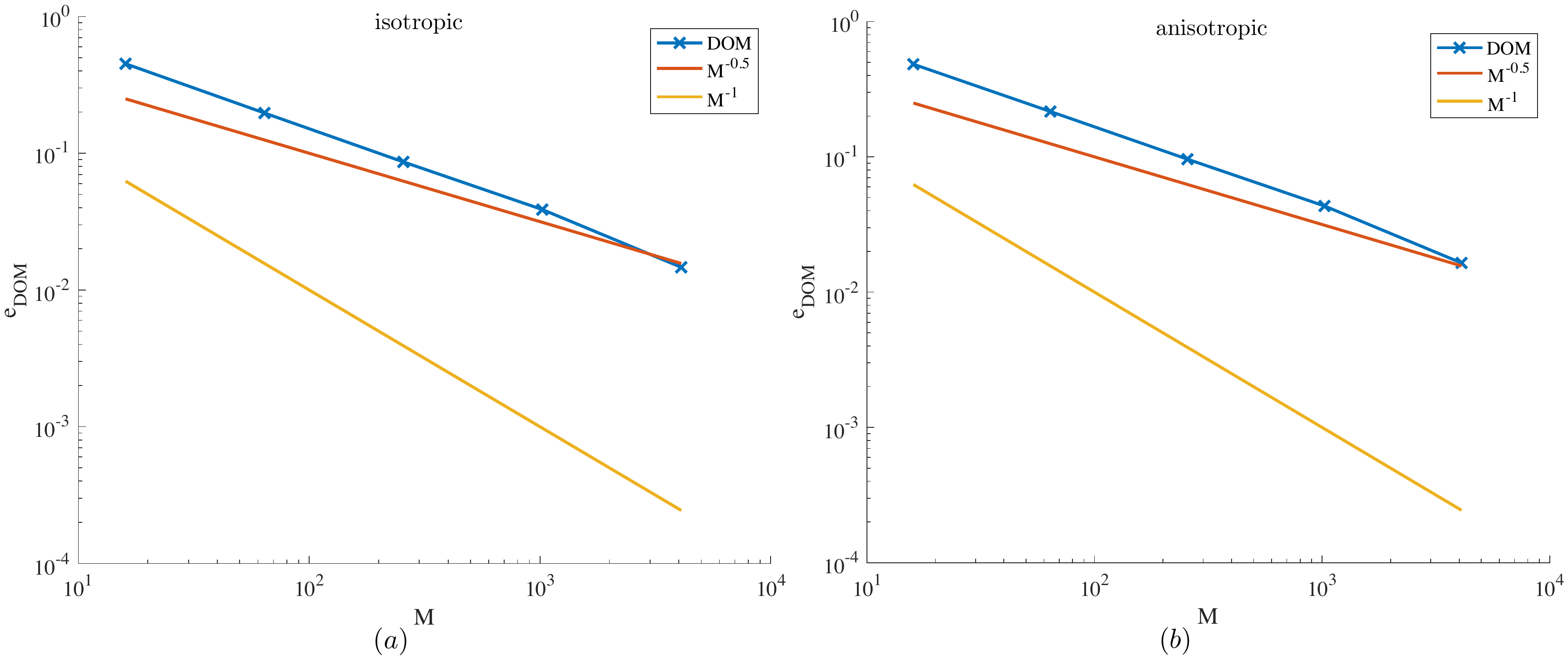}
	\caption{Example 1. The convergence order of DOM w.r.t. $M$.(a) isotropic scattering; (b) anisotropic scattering.}
	\label{2Dc}
\end{figure}

\paragraph{Convergence order of DOM and comparison of complexity. }
Since there is discontinuities in the inflow boundary conditions, the convergence order of DOM is low. The relative discrete $\ell^2$ error for DOM is defined by
\blue{\begin{equation} e^M_{DOM}=\|\phi_{0,M}^{(N)}-\phi_{0}^{ref}\|_2\Big/\|\phi_{0}^{ref}\|_2,\label{eDOM}\end{equation}
where $N$ is defined in \eqref{defN}; $\phi_{0,M}^{(N)}$ is the solution of DOM with $M$ ordinates and has converged as in \eqref{defN};} $\phi_0^{ref}$ is the reference solution to equation \eqref{2deqdc} obtained with a fine quadrature $M_{ref}=2^{14}$. Figure \ref{2Dc} shows the convergence orders of DOM with respect to $M$ for both isotropic and anisotropic scattering cases. They range between 0.5 and 0.6, which are quite low.

Note that the computational cost of DOM w.r.t.$M$ is $\mathcal{O}(M^2)$, and the cost of RSI w.r.t. the number of samples $S$ is $\mathcal{O}(S)$. Since the convergence order of RSI w.r.t. $S$ is 0.5, when the convergence order of DOM w.r.t. $M$ is lower than 1, to achieve the same level of accuracy, the overall complexity of RSI is lower than that of DOM.

\begin{figure}[!htbp]
	\centering
	\includegraphics[width=0.9\linewidth]{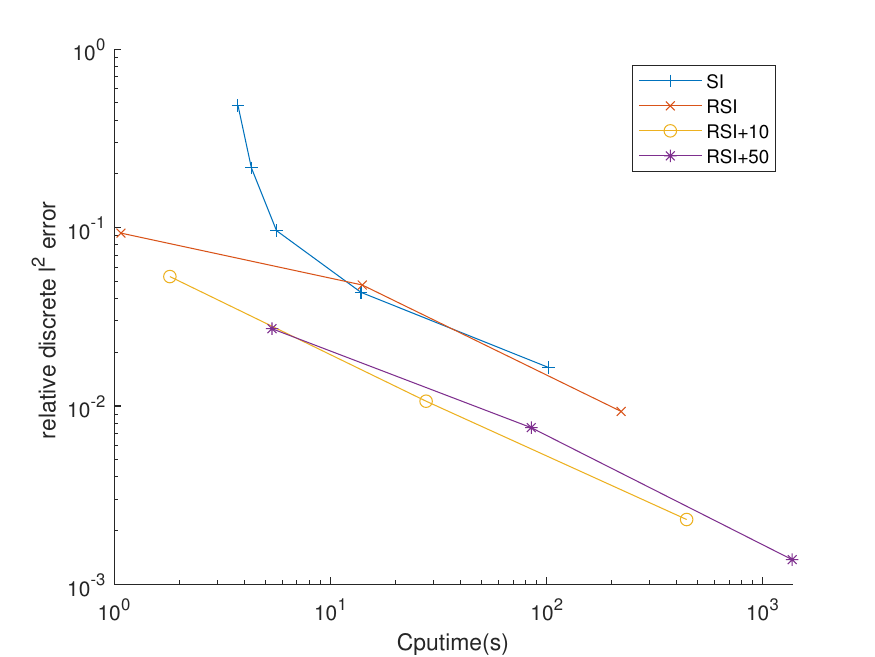}
	\caption{Example 1. The relative discrete $\ell^2$ error versus CPU time cost for SI, RSI, RSI averaged from from $N_{itr}$ to $N_{itr}+10$ and RSI averaged from $N_{itr}$ to $N_{itr}+50$ in the anisotropic scattering case.}
	\label{fig:errorvstime}
\end{figure}

\blue{For a more quantitative comparison, we compare the relative discrete $\ell^2$ error versus the CPU time spent by SI and RSI for the anisotropic scattering. Because RSI uses a finer quadrature than SI, we compare the obtained numerical results to the reference solution computed using DOM with a very fine mesh $M = 2^{14}$.}

\blue{
Define
\begin{equation}
e_{\text{RSI}} = \frac{\|\bar{\phi}_{0}^{(N)} - \phi_{0}^{\text{ref}}\|_2}{\|\phi_{0}^{\text{ref}}\|_2}, \label{eRSI}
\end{equation}
where $\bar{\phi}_{0}^{(N)}$ is the average of $S$ samples.
We compare the performance of SI and RSI on a Lenovo Legion R9000P-2021H laptop using Matlab coding. The CPU times for only the iteration part are measured for both SI and RSI. For SI with different $M$, the iterations stop when the convergence condition \eqref{defN} is satisfied. For RSI, we use $M = 2^{14}$ ordinates and $N_{\text{itr}} = 8$ iterations, with which the convergence condition \eqref{defN} is satisfied when the DOM system with $2^{14}$ ordinates is solved with SI. Although RSI is easy to parallelize, we choose serial processing for all $S$ samples here to make the comparison of complexity more fair.}

\blue{
In Figure \ref{fig:errorvstime}, we show the CPU time for different DOM system sizes $M = 2^{4}, 2^6, 2^8, 2^{10}, 2^{12}$ when performing SI, and their corresponding errors $e_{\text{DOM}}$. As a comparison, we also run RSI with $S = 2^8, 2^{12}, 2^{16}$ experiments and show the CPU time and their corresponding errors $e_{\text{RSI}}$. Moreover, due to the ergodicity, we also test RSI by averaging samples from iteration $N_{\text{itr}}$ to $N_{\text{itr}} + 10$ and from $N_{\text{itr}}$ to $N_{\text{itr}} + 50$ for all $S$ experiments. The results are plotted in Figure \ref{fig:errorvstime}. Clearly, to achieve the same accuracy, the overall CPU times of RSI are comparable to those of SI, while RSI is very easy to parallelize. Furthermore, performance can be improved by taking additional samples after iteration $N_{\text{itr}}$, demonstrating the benefits of ergodicity.}

\paragraph{Ray effect mitigation.} In the case of isotropic scattering, the particle density obtained by DOM with a fine quadrature $M=2^{14}$ is depicted in Figure \ref{RE_DD_fine}. When $M$ is not sufficiently large, the particle density exhibits nonphysical oscillations, as shown in Figure \ref{RE_DD_coarse}. These oscillations, which cannot be diminished by refining the spatial discretization, are referred to as ray effects. To quantify the ray effect, we define the following metric
\eq$e^{\infty}=\max_{i,j}\left\{\left|\phi_{0,i,j}-\phi_{0,i,j}^{ref}\right|\right\}\Big/\max_{i,j}\left\{\left|\phi_{0,i,j}^{ref}\right|\right\},\label{einfty}$
where $\phi_{0,i,j}^{ref}$ corresponds to the reference solution with a fine quadrature $M^{ref}=2^{14}$. The source iteration converges in \blue{$N_{itr}=16$} steps with tolerance $10^{-10}$ and $e^{\infty}=0.227$ for $M=2^4$ as in Figure \ref{RE_DD_coarse}, which is considerably large. Upon applying the RSI with $M=2^{14}$, as illustrated in Figure \ref{RE_DDR_coarse}, even with only $S=2^8$ samples, the ray effect has been significantly mitigated. By substituting $\phi_{0,i,j}$ with $\bphi_{0,i,j}$ in equation \eqref{einfty}, the ray effect of RSI can be quantified, yielding 
$e^{\infty}=0.081$. The computational complexity of RSI with $S=2^8$ samples is comparable to that of the SI method using $M=2^4$ ordinates, which is very low. Thanks to the ergodicity, one can take the average from \blue{$N_{itr}=16$ to $N_{itr}+10=26$} to generate an even larger number of samples $S=2^8\times 10$ (not totally independent due to the evolution of Markov chain), $e^{\infty}$ can be further reduced to $0.022$, as demonstrated in Figure \ref{RE_DDR_fine}.

\begin{figure}[!htbp]
	\centering
	\subfloat[]{\includegraphics[width=0.4\linewidth]{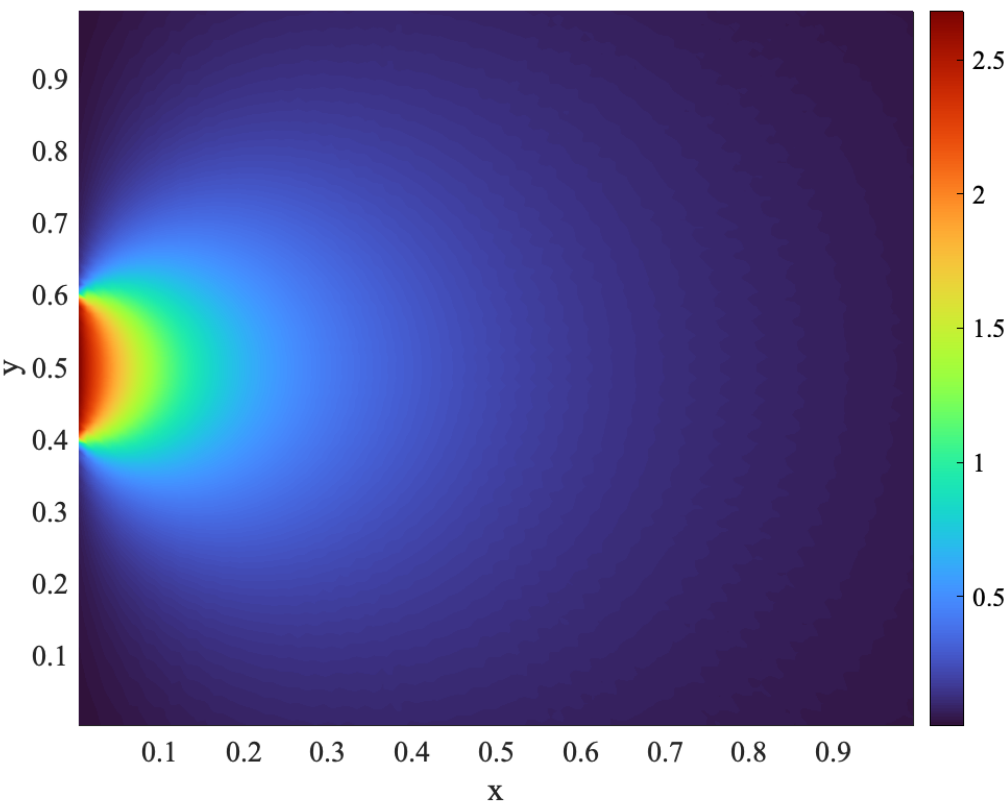}\label{RE_DD_fine}}
	\subfloat[]{\includegraphics[width=0.4\linewidth]{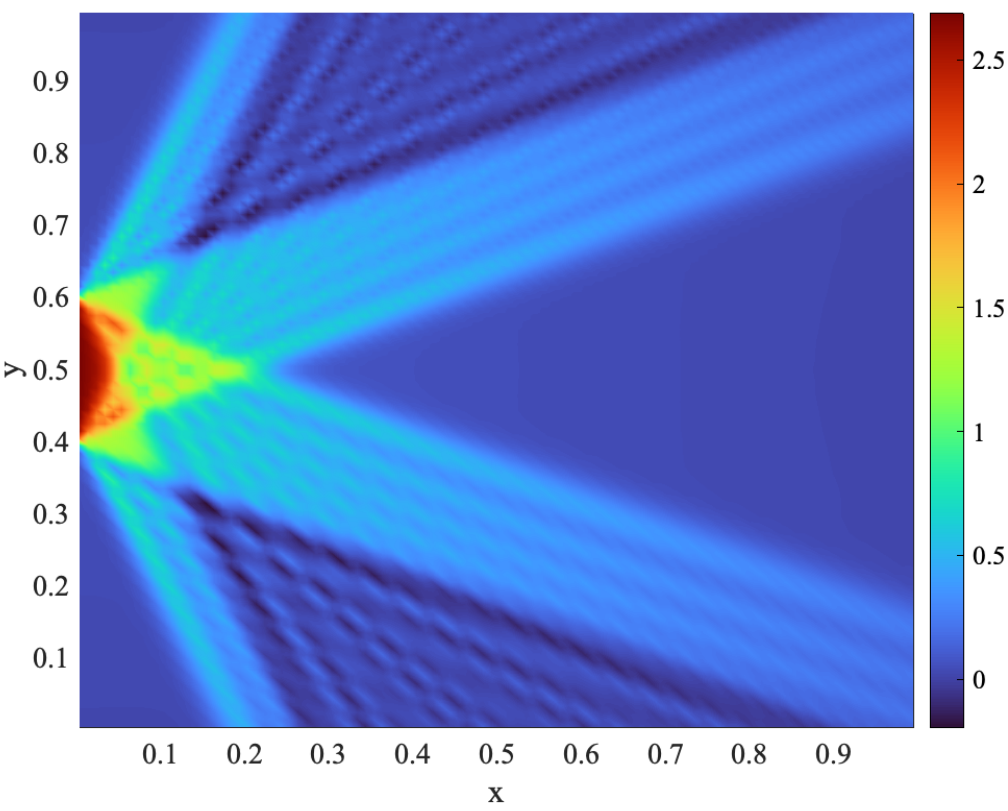}\label{RE_DD_coarse}}\\
	\subfloat[]{\includegraphics[width=0.4\linewidth]{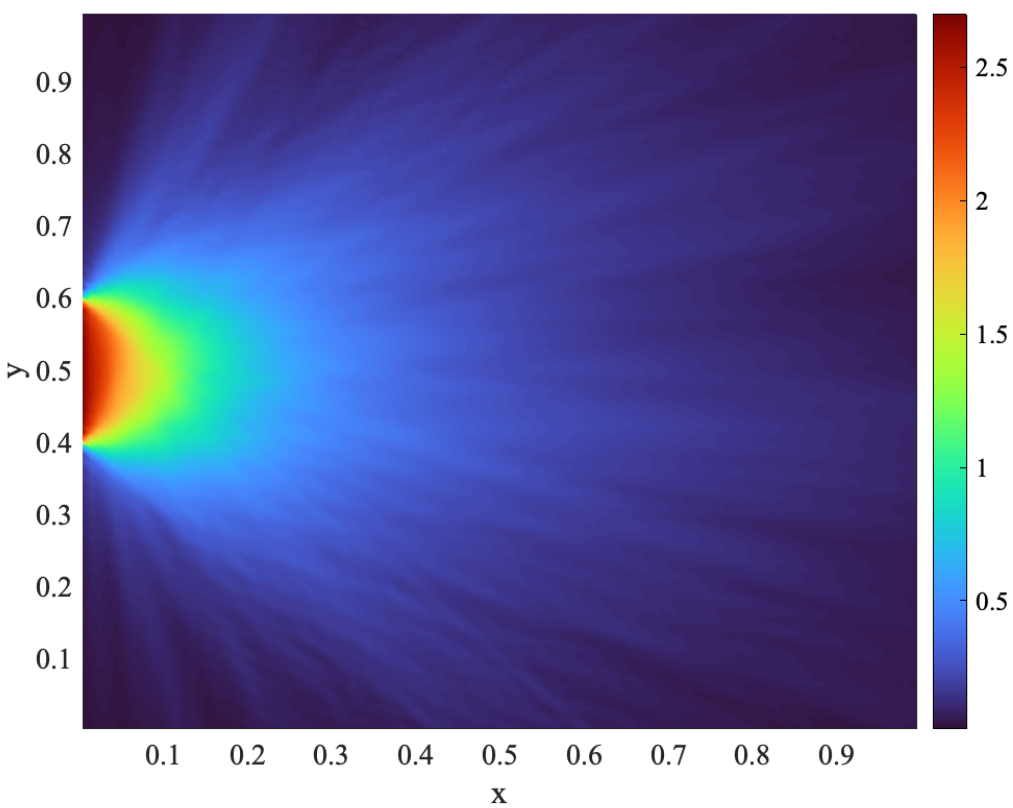}\label{RE_DDR_coarse}}
    \subfloat[]{\includegraphics[width=0.4\linewidth]{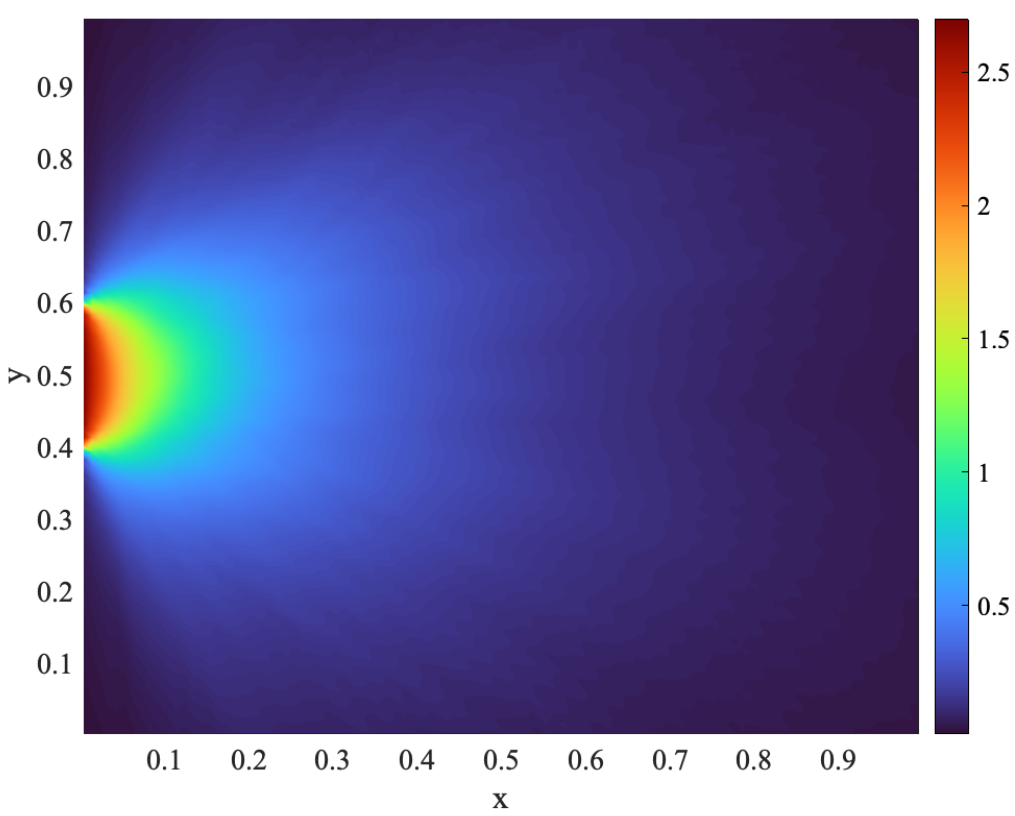}\label{RE_DDR_fine}}
	\caption{Example 1. The particle density $\phi_{0,i,j}$ obtained numerically with (a) SI with fine quadrature ($M=2^{14}$); (b) SI with coarse quadrature ($M=2^4$); (c) RSI with $M=2^{14}$ and $S=2^8$ and the same number of iterations (d) RSI with $M=2^{14}$ and $S=2^{8}$ but the average is taken from $N_{itr}$ to $N_{itr}+10$.}
	\label{2DRE}
\end{figure}

\paragraph{Convergence order of RSI w.r.t. $G$.} When more than one ordinate is chosen in each iteration step, the computational cost of running one sample increases, but the variance among different samples can be reduced. One can examine the convergence order to see how the variance decreases with respect to the number of ordinates $G$ used in each iteration step. As shown in Figure \ref{2DisowrtG}, the values of $e_{RSI}^{(N)}$, as defined in \eqref{ersi}, are displayed. These are obtained using different values of $G$ while keeping the same $M=2^{14}$, $S=8$. The convergence order with respect to $G$ is $1$, which is consistent with the results discussed in Section 3.1 for the isotropic case. Similar outcomes can be observed for the anisotropic case.

\begin{figure}[!htbp]
	\centering
	\includegraphics[width=0.9\linewidth]{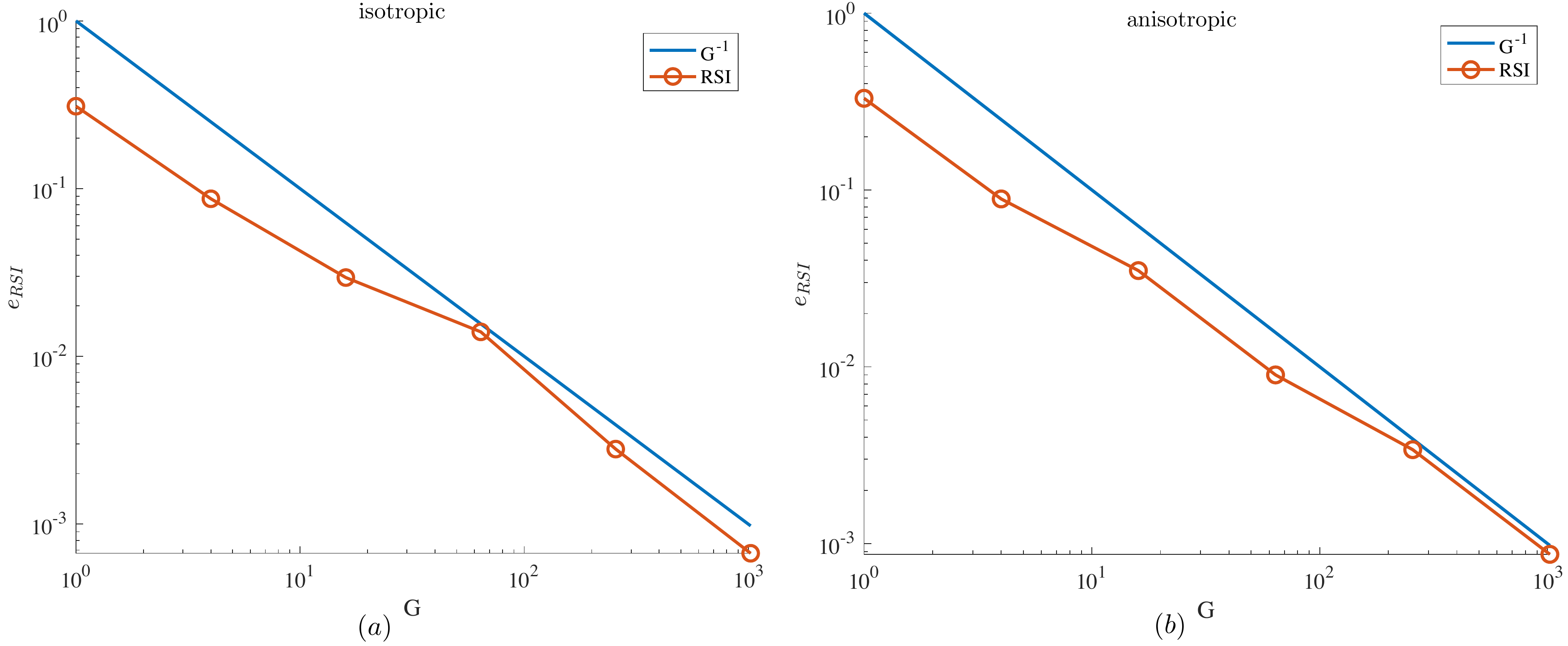}
	\caption{Example 1. The convergence order of $e_{RSI}^{(N)}$ w.r.t the number of ordinates used in each iteration step $G$, with fixed $M=2^{14}$, $S=8$.  (a) isotropic scattering; (b) anisotropic scattering.}
	\label{2DisowrtG}
\end{figure}


\paragraph{Example 2.} To demonstrate that the RSI method yields satisfactory results independent of the problem setup, we test a more complex lattice problem whose parameters are identical to Example 6.2 in \cite{camminady2019ray}. The problem features no inflow but includes a source within a heterogeneous material.

\begin{figure}
    \centering
    \begin{tikzpicture}
        \draw (0,0)--(0,5.6)--(5.6,5.6)--(5.6,0)--(0,0);
        \filldraw[draw=red,fill=red] (0.8,0.8) rectangle (1.6,1.6);
        \filldraw[draw=red,fill=red] (0.8,2.4) rectangle (1.6,3.2);
        \filldraw[draw=red,fill=red] (0.8,4) rectangle (1.6,4.8);
        \filldraw[draw=red,fill=red] (2.4,0.8) rectangle (3.2,1.6);
        \filldraw (2.4,2.4) rectangle (3.2,3.2);
        \filldraw[draw=red,fill=red] (4,0.8) rectangle (4.8,1.6);
        \filldraw[draw=red,fill=red] (4,2.4) rectangle (4.8,3.2);
        \filldraw[draw=red,fill=red] (1.6,1.6) rectangle (2.4,2.4);
        \filldraw[draw=red,fill=red] (3.2,1.6) rectangle (4,2.4);
        \filldraw[draw=red,fill=red] (1.6,3.2) rectangle (2.4,4);
        \filldraw[draw=red,fill=red] (3.2,3.2) rectangle (4,4);
        \filldraw[draw=red,fill=red] (4,4) rectangle (4.8,4.8);
    \end{tikzpicture}
    \caption{Example 2. Distribution of cross sections and sources in lattice problem}
    \label{lattice}
\end{figure}
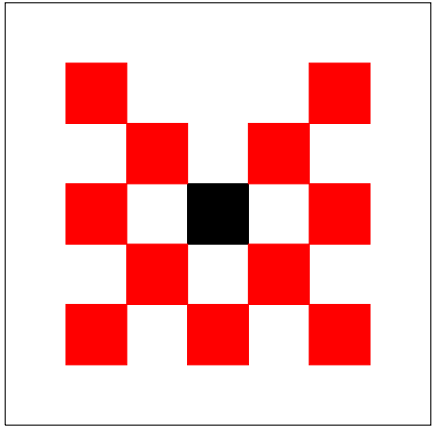

As depicted in Figure \ref{lattice}, the spatial domain is defined as $[0,7]\times[0,7]$ and is uniformly divided into $7\time 7$ squares. The red or black area is strongly absorbing medium in which $\Sigma_T=10, \Sigma_S=0$ (i.e., $\Sigma_a=10$),
and the white area is strongly scattering background in which 
$\Sigma_T=\Sigma_S=1$ (i.e., $\Sigma_a=0$). Within the black square, there is an isotropic source with $Q=1$, while $Q=0$ elsewhere.  In this example, both isotropic and anisotropic scatterings are tested.

The method of measuring the convergence of RSI is analogous to that used in Example 1. Figure \ref{2DisowrtS2} illustrates the convergence order with respect to the number of samples $S$ for both isotropic and anisotropic scattering, where we use $I=J=70$ for spacial discretization. In this more complex lattice problem, the convergence order is $0.5$ for all cases.

\begin{figure}[!htbp]
	\centering
	\subfloat[]{\includegraphics[width=0.45\linewidth]{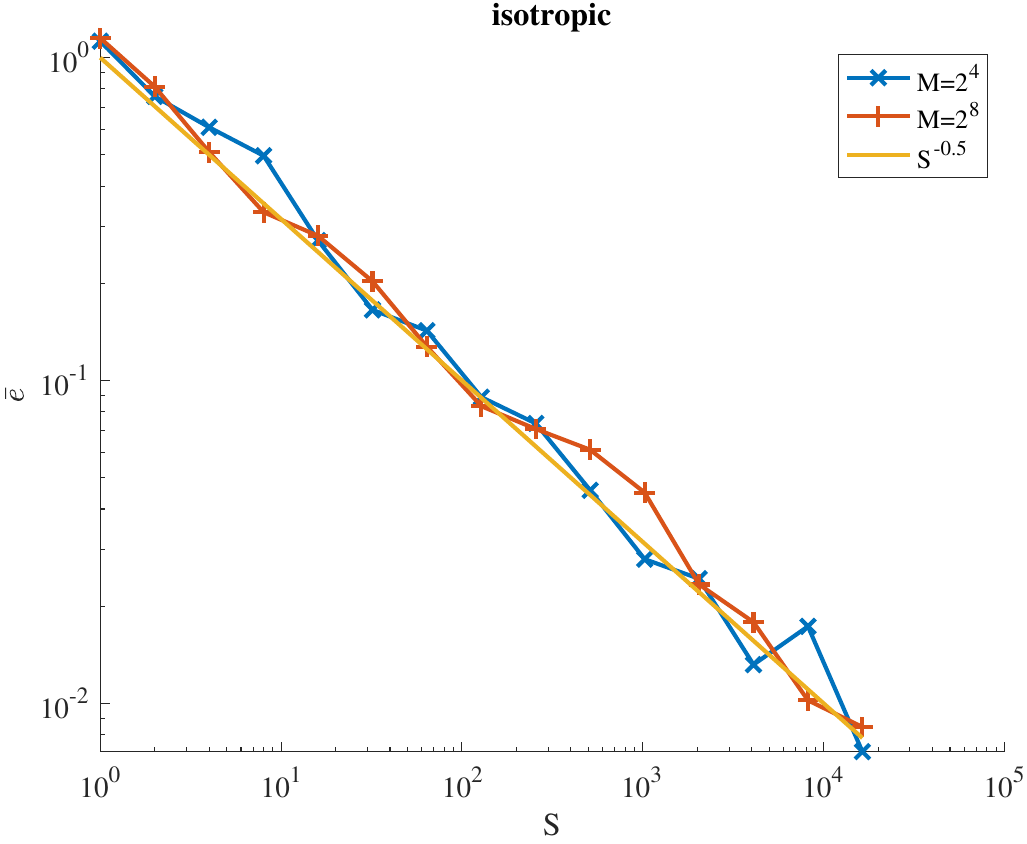}\label{errmSisoex2}}
	\subfloat[]{\includegraphics[width=0.45\linewidth]{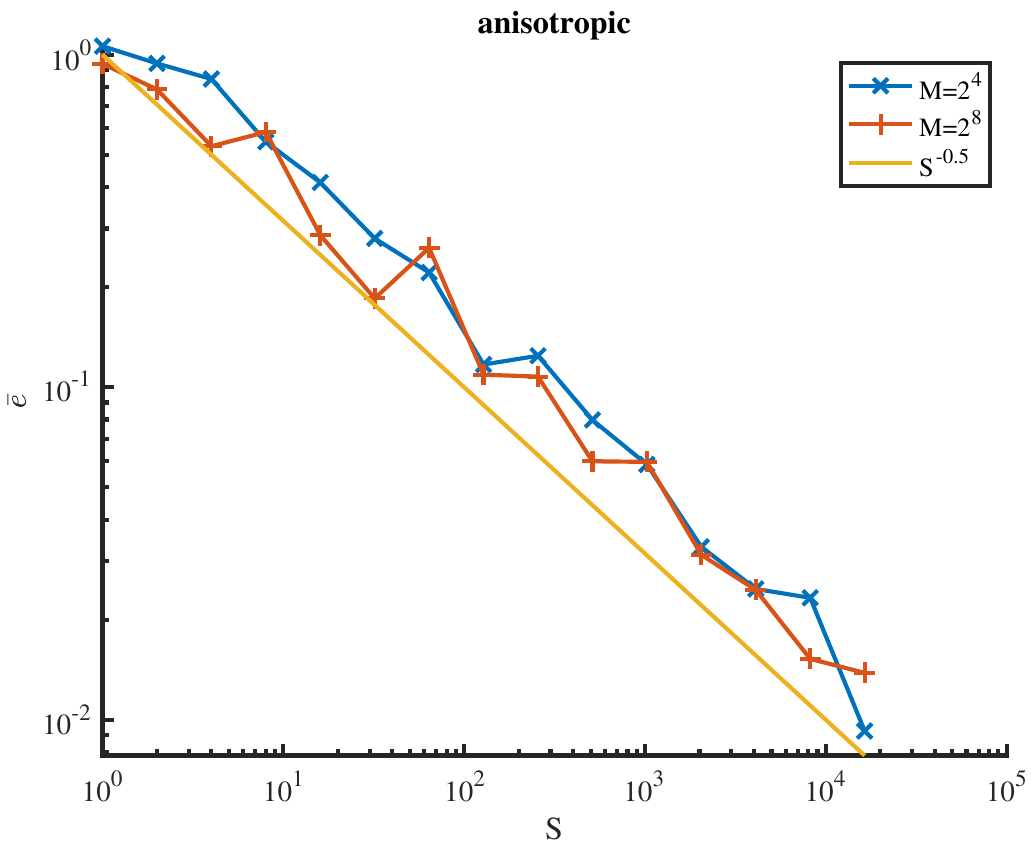}\label{errmSanisoex2}}
	\caption{Example 2. The relative $\ell^2$ error $\bar{e}$ w.r.t number of chain $S$, with $M=2^4, 2^8$ and (a) isotropic scattering; (b) anisotropic scattering.}
	\label{2DisowrtS2}
\end{figure}

To demonstrate more clearly the mitigation of ray effect,  we use $I=J=280$ for Fig. 8. For isotropic scattering, the numerical solutions of DOM with a fine quadrature $M=2^{14}$ and coarse quadrature $M=2^4$ are respectively displayed in Figures \ref{lattice_solution} (a) (b) and \ref{lattice_solution} (c) (d). \blue{The source iteration for the coarse quadrature converges in $N_{itr}=17$ steps in this case.} We use $e^\infty$ in \eqref{einfty} to quantify the ray effect. We find $e^{\infty}=0.090$ for $M=2^4$. Upon applying RSI, as depicted in Figure \ref{lattice_solution} (e) (f), even with only $S=2^8$ samples and the same number of iteration \blue{$N_{itr}=17$} (comparable in complexity to coarse quadrature DOM), the ray effect has been significantly mitigated, yielding $e^{\infty}=0.031$. Due to the ergodicity, one can take the average of the tail data, $e^{\infty}$ can be further reduced. As shown in Figure \ref{lattice_solution} (g) (h), when $S=2^8$ but the average is taken for \blue{$N_{itr}=17$ to $N_{itr}+40=57$}, $e^{\infty}$ is reduced to 0.011.
\begin{figure}[!htbp]
	\centering
 \includegraphics[width=0.86\linewidth]{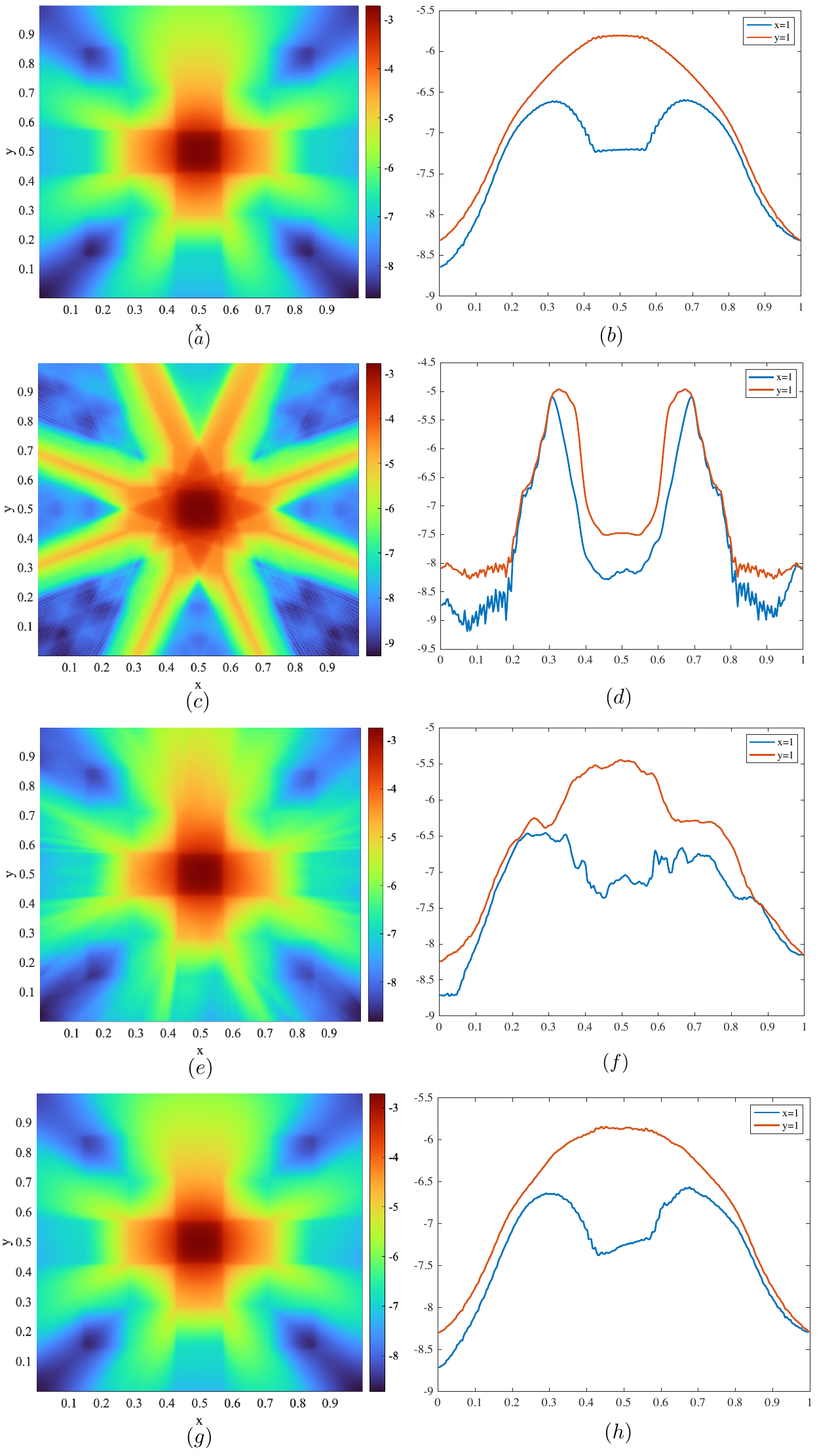}
	\caption{Example 2: logarithimic value of the approximated $0^{th}$ order moment $\phi_{0,i,j}$ for (a,b) source iteration with fine quadrature ($M=2^{14}$); (c,d) source iteration with coarse quadrature ($M=2^4$); (e,f) random source iteration with $M=2^{14}$, $S=2^8$ and the same number of iteration (g,h) random source iteration with $M=2^{14}$, $S=2^{8}$ but the average is taken for $N_{iter}$ to $N_{iter}+40$.}
	\label{lattice_solution}
\end{figure}

\section{Conclusion and discussions}\label{sec:discussion}
In this paper, we introduce and analyze the RSI method. We have rigorously proved that RSI is unbias w.r.t. to the SI method and  that its variance is uniformly bounded across iteration steps. We view the iterates of RSI method as the states of a Markov chain and give the convergence rate of the laws (distributions) of the states to the invariant measure.

RSI can mitigate the ray effect without increasing the overall computational cost and has the benefit of being easy to parallelize.
A lot of parallel codes have been developed based on the SI, 
for example, as seen in \cite{chen2017ares, zhao2007parallel} and the references therein. The standard approach involves using domain decomposition and a carefully designed sweeping process for the spatial transport operator, as detailed in \cite{chen2017ares}. The specifics of scheduling may vary depending on different spatial geometries and types of spatial meshes, as discussed in \cite{adams2020provably}. If one sample is assigned to one processor, there is no information exchange between different ordinates, thus fully parallelizing the angular variable. It is straightforward to couple RSI with spatial parallelism; however, designing the most efficient coupling strategy will be the subject of our future work.

The SI method converges slowly for optically thick materials, as noted in \cite{adams2002fast}. To accelerate the convergence rate of the SI in such cases, various acceleration techniques have been proposed. It is interesting to investigate how to couple RSI  with these acceleration techniques, such as Diffusion-Synthetic Acceleration (DSA), in order to effectively handle optically thick cases.  \blue{Moreover, modern iteration methods rely on preconditioned GMRES. One of the main points is that the iteration is not performed on the ordinates, but on the moments of the solution. It would be interesting to investigate whether the key idea of RSI can be extended to state-of-the-art linear solvers.}
 

\appendix


	\bibliographystyle{abbrvnat}
	\bibliography{main_template}

\begin{thebibliography}{10}

\bibitem{adams2002fast}
Marvin~L Adams and Edward~W Larsen.
\newblock Fast iterative methods for discrete-ordinates particle transport
  calculations.
\newblock {\em Progress in nuclear energy}, 40(1):3--159, 2002.

\bibitem{adams2020provably}
Michael~P Adams, Marvin~L Adams, W~Daryl Hawkins, Timmie Smith, Lawrence
  Rauchwerger, Nancy~M Amato, Teresa~S Bailey, Robert~D Falgout, Adam Kunen,
  and Peter Brown.
\newblock Provably optimal parallel transport sweeps on semi-structured grids.
\newblock {\em Journal of Computational Physics}, 407:109234, 2020.

\bibitem{benaim2022markov}
Michel Bena{\"\i}m and Tobias Hurth.
\newblock {\em Markov Chains on Metric Spaces: A Short Course}.
\newblock Springer, 2022.

\bibitem{camminady2019ray}
Thomas Camminady, Martin Frank, Kerstin K{\"u}pper, and Jonas Kusch.
\newblock Ray effect mitigation for the discrete ordinates method through
  quadrature rotation.
\newblock {\em Journal of Computational Physics}, 382:105--123, 2019.

\bibitem{chen2017ares}
Yixue Chen, Bin Zhang, Liang Zhang, Junxiao Zheng, Ying Zheng, Cong Liu, et~al.
\newblock Ares: a parallel discrete ordinates transport code for radiation
  shielding applications and reactor physics analysis.
\newblock {\em Science and Technology of Nuclear Installations}, 2017, 2017.

\bibitem{colomer2013parallel}
Guillem Colomer, Rick Borrell, Francesc~Xavier Trias, and I~Rodr{\'\i}guez.
\newblock Parallel algorithms for sn transport sweeps on unstructured meshes.
\newblock {\em Journal of Computational Physics}, 232(1):118--135, 2013.

\bibitem{Dai2023}
T.~Dai, L.~F. Xu, B.~W Li, H.~Y. Shen, and X.~M. Shi.
\newblock A first collision source method using spherical harmonics method for
  arbitrary shape and order finite element mesh.
\newblock {\em Progress in Nuclear Energy}, 165:104929, 2023.

\bibitem{de2019quasi}
Pedro~Henrique de~Almeida~Konzen, Leonardo~Fernandes Guidi, and Thomas Richter.
\newblock Quasi-random discrete ordinates method for neutron transport
  problems.
\newblock {\em Annals of Nuclear Energy}, 133:275--282, 2019.

\bibitem{dedner2002adaptive}
Andreas Dedner and Peter Vollm{\"o}ller.
\newblock An adaptive higher order method for solving the radiation transport
  equation on unstructured grids.
\newblock {\em Journal of Computational Physics}, 178(2):263--289, 2002.

\bibitem{giani2016hp}
Stefano Giani and Mohammed Seaid.
\newblock hp-adaptive discontinuous galerkin methods for simplified pn
  approximations of frequency-dependent radiative transfer.
\newblock {\em Computer Methods in Applied Mechanics and Engineering},
  301:52--79, 2016.

\bibitem{hairer2011yet}
Martin Hairer and Jonathan~C Mattingly.
\newblock Yet another look at harris’ ergodic theorem for markov chains.
\newblock In {\em Seminar on Stochastic Analysis, Random Fields and
  Applications VI: Centro Stefano Franscini, Ascona, May 2008}, pages 109--117.
  Springer, 2011.

\bibitem{Konzen2023}
P.~H.~A. KONZEN, L.~F. GUIDI, and T.~RICHTER.
\newblock Quasi-random discrete ordinates method to radiative transfer equation
  with linear anisotropic scattering.
\newblock {\em Defect and Diffusion Forum}, 427:109--119, 2023.

\bibitem{Lathrop1971}
K.~D. Lathrop.
\newblock Remedies for ray effects.
\newblock {\em Nucl. Sci. Eng.}, 45:255--168, 1971.

\bibitem{lathrop1968ray}
Kaye~D Lathrop.
\newblock Ray effects in discrete ordinates equations.
\newblock {\em Nuclear Science and Engineering}, 32(3):357--369, 1968.

\bibitem{ROM}
Min~Tang Lei~Li and Yuqi Yang.
\newblock Random ordinate method for mitigating the ray effect in radiative
  transport equation simulations.
\newblock {\em submitted}, 2024.

\bibitem{LewisMiller}
E.~E. Lewis and Jr. W.~F.~Miller.
\newblock {\em Computational methods of neutron transport}.
\newblock Wiley-Interscience, 1993.

\bibitem{liu2010analysis}
Jian-Guo Liu and Luc Mieussens.
\newblock Analysis of an asymptotic preserving scheme for linear kinetic
  equations in the diffusion limit.
\newblock {\em SIAM Journal on Numerical Analysis}, 48(4):1474--1491, 2010.

\bibitem{morel1996linear}
Jim~E Morel, Todd~A Wareing, and Kenneth Smith.
\newblock A linear-discontinuous spatial differencing scheme forsnradiative
  transfer calculations.
\newblock {\em Journal of Computational Physics}, 128(2):445--462, 1996.

\bibitem{pasmann2023quasi}
Sam Pasmann, Ilham Variansyah, CT~Kelley, and Ryan McClarren.
\newblock A quasi--monte carlo method with krylov linear solvers for multigroup
  neutron transport simulations.
\newblock {\em Nuclear Science and Engineering}, 197(6):1159--1173, 2023.

\bibitem{santambrogio2015optimal}
Filippo Santambrogio.
\newblock Optimal transport for applied mathematicians.
\newblock {\em Birk{\"a}user, NY}, 55(58-63):94, 2015.

\bibitem{sheng2021uniform}
Qiwei Sheng and Cory Hauck.
\newblock Uniform convergence of an upwind discontinuous galerkin method for
  solving scaled discrete-ordinate radiative transfer equations with isotropic
  scattering.
\newblock {\em Mathematics of Computation}, 90(332):2645--2669, 2021.

\bibitem{spanier1959monte}
Jerome Spanier.
\newblock {\em Monte Carlo methods and their application to neutron transport
  problems}, volume 195.
\newblock Bettis Atomic Power Laboratory, 1959.

\bibitem{Tencer2016}
J.~Tencer.
\newblock Ray effect mitigation through reference frame rotation.
\newblock {\em J. Heat Transfer.}, 38:112701, 2016.

\bibitem{wareing1998first}
Todd~A Wareing, Jim~E Morel, and Donald~K Parsons.
\newblock A first collision source method for attila, an unstructured
  tetrahedral mesh discrete ordinates code.
\newblock Technical report, Los Alamos National Lab.(LANL), Los Alamos, NM
  (United States), 1998.

\bibitem{centuryreview}
E.~Sartori Y.~Azmy.
\newblock {\em Nuclear Computational Science--A Century in Review}.
\newblock Springer, 2010.

\bibitem{zhang2018goal}
Bin Zhang, Liang Zhang, Cong Liu, and Yixue Chen.
\newblock Goal-oriented regional angular adaptive algorithm for the sn
  equations.
\newblock {\em Nuclear Science and Engineering}, 189(2):120--134, 2018.

\bibitem{GaoZhao2009}
H.~Gao; H.~K. Zhao.
\newblock A fast-forward solver of radiative transfer equation.
\newblock {\em Transport Theory and Statistical Physics}, 38:149--192, 2009.

\bibitem{zhao2007parallel}
Hongkai Zhao.
\newblock Parallel implementations of the fast sweeping method.
\newblock {\em Journal of Computational Mathematics}, pages 421--429, 2007.

\end{thebibliography}
\end{document}